\documentclass[11pt,letterpaper]{amsart}
\usepackage{times}
\usepackage{amsfonts}
\usepackage{amsthm}
\usepackage{amsmath}
\usepackage{amscd}
\usepackage[latin2]{inputenc}
\usepackage{t1enc}
\usepackage[mathscr]{eucal}
\usepackage{indentfirst}
\usepackage{graphicx}
\usepackage{graphics}
\usepackage{enumitem}
\usepackage{url}
\usepackage{pict2e}
\usepackage{pgfplots}
\usepackage{epic}
\numberwithin{equation}{section}
\usepackage[left=1in, right=1in, top=1.5in, bottom=1.5in]{geometry}
\usepackage{epstopdf} 
\usepackage{tikz-cd}
\usepackage{lipsum}
\usepackage{hyperref}
\usepackage{mathrsfs}

\hypersetup{colorlinks,linkcolor={blue},citecolor={blue},urlcolor={blue}}

\newtheorem{theorem}{Theorem}[section]
\newtheorem{corollary}[theorem]{Corollary}
\newtheorem{lemma}[theorem]{Lemma}
\newtheorem{proposition}[theorem]{Proposition}
\theoremstyle{definition}
\newtheorem{definition}[theorem]{Definition}
\newtheorem{example}[theorem]{Example}

\newtheorem{remark}[theorem]{Remark}

\makeatletter
\newcommand{\extp}{\@ifnextchar^\@extp{\@extp^{\,}}}
\def\@extp^#1{\mathop{\bigwedge\nolimits^{\!#1}}}
\makeatother

\makeatletter
\newcommand*\bigcdot{\mathpalette\bigcdot@{.5}}
\newcommand*\bigcdot@[2]{\mathbin{\vcenter{\hbox{\scalebox{#2}{$\m@th#1\bullet$}}}}}
\makeatother

\newlist{steps}{enumerate}{1}
\setlist[steps, 1]{label = Step \arabic*.}

\begin{document}

\title{Symmetries and Vanishing Theorems for Symplectic Varieties}

\author[B.Tighe]{Benjamin Tighe}
\address{Department of Mathematics, University of Oregon}
\email{bentighe@uoregon.edu}

\maketitle

\begin{abstract}
We describe the local and Steenbrink vanishing problems for singular symplectic varieties with isolated singularities. We do this by constructing a morphism $$\mathbb D_X(\underline \Omega_X^{n+p}) \to \underline \Omega_X^{n+p}$$ for a symplectic variety $X$ of dimension $2n$ for $\frac{1}{2}\mathrm{codim}_X(X_{\mathrm{sing}}) < p$, where $\underline \Omega_X^k$ is the $k^{th}$-graded piece of the Du Bois complex and $\mathbb D_X$ is the Grothendieck duality functor.  We show this morphism is a quasi-isomorphism when $p = n-1$ and that this symmetry descends to the Hodge filtration on the intersection Hodge module.  As applications, we describe the higher Du Bois and higher rational properties for symplectic germs and the cohomology of primitive symplectic 4-folds.
\end{abstract}

\tableofcontents

\section{Introduction}

Let $X$ be a complex algebraic variety.  A good way to understand how well-behaved the singularities of $X$ are is to look for properties which hold for smooth varieties. For instance, we say $X$ has \textit{rational singularities} if \begin{equation} \label{equation rational singularities}
    \mathscr O_X \xrightarrow{\sim} \mathbf R\pi_*\mathscr O_{\widetilde X}
\end{equation} for any resolution of singularities $\pi$, generalizing Hironaka's solution \cite[pp. 144-145]{hironaka1964resolution} to a question of Grothendieck \cite[Problem B]{grothendieck1960cohomology}.  Rational singularities are a broad class which include the (Kawamata) log-terminal singularities, and the utility of these singularities is well-known in birational geometry.  There are also the Hodge theoretic singularities coming from Deligne's construction of the mixed Hodge structure \cite{deligne1974theorie}, \cite{du1981complexe}. 
 Associated to a complex algebraic variety $X$ is the \textit{Du Bois} complex $(\underline \Omega_X^\bullet, F)$, a filtered object whose graded quotients $\underline \Omega_X^p : = \mathrm{gr}_F^p\underline \Omega_X^{\bullet}[p]$ are objects $D^b(X)$.  It generalizes the holomorphic de Rham complex for complex manifolds (in particular, $\Omega_X^p \cong \underline \Omega_X^p$ if $X$ is smooth).  We say $X$ has \textit{Du Bois singularities} if $$\mathscr O_X \xrightarrow{\sim}\underline \Omega_X^0,$$ generalizing the morphism (\ref{equation rational singularities}).  Not only are rational singularities Du Bois \cite[Corollary 2.6]{kovacs1999rational}, \cite[5.4. Theorem]{saito2000mixed}, but they also include the log-canonical singularities \cite[Theorem 1.4]{kollar2010log}, leading to many useful applications in the study of higher dimensional varieties and moduli spaces.

In a different direction, one can measure a singularity by the kinds of vanishing theorems it exhibits.  For any normal complex algebraic variety, there is the Grauert-Riemenschneider vanishing theorem, which states $\mathbf R\pi_*\omega_{\widetilde X} \cong_{\mathrm{qis}} \pi_*\omega_{\widetilde X}$ for any resolution of singularities $\pi$ \cite{grauert1970}.  A vast generalization of this is the Steenbrink vanishing theorem \cite[Theorem 2]{steenbrink1985vanishing}: if $\pi:\widetilde X \to X$ is a \textit{log}-resolution of singularities, then \begin{equation} \label{equation steenbrink vanishing}
   R^j\pi_*\Omega_{\widetilde X}^k(\log E)(-E) : =  R^j\pi_*\Omega_{\widetilde X}^k(\log E)\otimes \mathscr I_E = 0, \quad j + k > \dim X,
\end{equation} where $\mathscr I_E$ is the ideal sheaf of the exceptional divisor $E$.  The vanishing does not hold for $j + k \le \dim X$, but usually the (non-)vanishing of these sheaves reveals something about the singularities.  For example, the edge case
\begin{equation} \label{equation steenbrink vanishing extended}
    R^{\dim X-1}\pi_*\Omega_{\widetilde X}^1(\log E)(-E) = 0
\end{equation} holds whenever $X$ has Du Bois singularities \cite[Theorem 14.1]{greb2011differential}.  In particular, this vanishing holds when $X$ has rational singularities, but it is unclear if (\ref{equation steenbrink vanishing}) can be extended further.  Instead, rational singularities satisfy another kind of vanishing, which we refer to as a \textit{local vanishing theorem}, conjectured in \cite{mustactua2020local} and proved in \cite[Theorem 1.10]{kebekus2021extending}: \begin{equation}\label{equation local vanishing rational singularities}
    R^{\dim X-1}\pi_*\Omega_{\widetilde X}^1(\log E) = 0.
\end{equation} The (non-)vanishing of the cohomology sheaves $R^j\pi_*\Omega_{\widetilde X}^k(\log E)$ often corresponds to some intrinsic description of the variety $X$, leading to many interesting results for hypersurfaces \cite{mustata2021bois}, local complete intersections \cite{mustata2021hodge} toric varieties \cite{shen2024local}, and Calabi-Yau varieties \cite{friedman2024kliminal}. 

The extension of Steenbrink vanishing (\ref{equation steenbrink vanishing extended}) and the local vanishing theorem (\ref{equation local vanishing rational singularities}) are important pieces of the holomorphic extension property for rational and Du Bois singularities.  Given an algebraic variety $X$ with regular locus $U$, it is important to know for which $k$ is the natural inclusion \begin{equation}\label{equation holomorphic extension}
    \pi_*\Omega_{\widetilde X}^k \hookrightarrow j_*\Omega_U^k
\end{equation} is an isomorphism for every resolution of singularities $\pi:\widetilde X \to X$.  This problem has been extensively studied in terms of the Hodge theoretic singularities \cite{greb2011differential}, \cite{kebekus2021extending}, \cite{park2023bois}, \cite{tighe2023holomorphic} with many applications, including the existence of functorial pull-back for rational singularities \cite{kebekus2013pullback}, \cite{kebekus2021extending}.  However, holomorphic extension is typically a weak condition, as (\ref{equation holomorphic extension}) will occur in low degree for topological reasons \cite[Theorem]{flenner1988extendability}, and there are examples of bad (i.e., log-canonical) singularities for which holomorphic extension holds in all degrees \cite[Example A.4]{kebekus2021extending}.  

\subsection{Derived symmetries for symplectic varieties} \label{subsection derived symmetries}
In this paper, we investigate the interaction between Hodge theory and vanishing theorems for a class of singularities with a ``strong'' holomorphic extension property: symplectic singularities.  We say that a normal complex variety $X$ is a \textit{symplectic variety} if there is a closed, non-degenerate form $\sigma \in H^0(U, \Omega_{U}^2)$ on the regular locus $U$ which extends to a (possibly degenerate) 2-form $\widetilde \sigma$ on some resolution of singularities $\pi:\widetilde X \to X$.  Such singularities are Gorenstein and rational \cite[Proposition 1.3]{beauville2000symplectic}\footnote{Beauville quotes a result of Miles Reid, which seems to be Kempf's criterion: $X$ has rational singularities if and only if it is Cohen-Macaulay and $\pi_*\omega_{\widetilde X}$ is reflexive.}. This is quite strong: the extension of a non-degenerate 2-form implies that \textit{all} differential forms extend holomorphically by \cite[Corollary 1.8]{kebekus2021extending}. 

To motivate our main theorem, we give a suggestive proof of Beauville's result using a derived form of symmetry.  Since $\sigma$ extends holomorphically, we have a composition \begin{equation}\label{equation symplectic composition}\mathscr O_X \to \mathbf R\pi_*\mathscr O_{\tilde X} \xrightarrow{L_\sigma^n} \mathbf R\pi_*\omega_{\tilde X} \xrightarrow{\sim} \pi_*\omega_{\tilde X} \hookrightarrow j_*\omega_{U},\end{equation} where $2n$ is the dimension, and $L_\sigma$ is the morphism obtained by wedging with $\widetilde \sigma$ and applying $\mathbf R\pi_*$. By definition, this composition $\mathscr O_X \to j_*\omega_U$ is an isomorphism, and $\mathscr O_X \to \mathbf R\pi_*\mathscr O_{\widetilde X}$ has a left inverse.  Therefore, $X$ has rational singularities by \cite[Theorem 1.1]{kovacs2000characterization}.  Indeed, since $\mathscr O_X \cong \pi_*\omega_{\widetilde X}$, dualizing the composition (\ref{equation symplectic composition}) gives a quasi-isomorphism \begin{equation} \label{equation symplectic composition dual}
    \mathbf R\pi_*\mathscr O_{\widetilde X}\to \mathbf R\pi_*\omega_{\tilde X} \to \mathbf R\pi_*\mathscr O_{\tilde X}
\end{equation} Since the middle term has no higher cohomology, $R^j\pi_*\mathscr O_{\tilde X} = 0$ for $j > 0$.

A common opinion is that symplectic varieties are among the most well-behaved singular varieties, and the goal of this paper is investigate the following motivating question: how does the symmetry \begin{equation} \label{equation symplectic symmetry regular locus}
    \Omega_U^{n-p} \xrightarrow{\sim} \Omega_U^{n+p}
\end{equation} obtained by wedging with $\sigma$ extend to the Du Bois complex? Of course, it's natural to guess there is a quasi-isomorphism of the form $\underline \Omega_X^{n-p} \cong_{\mathrm{qis}} \underline \Omega_X^{n+p}$, and we note that this does hold in degree 0.  Surprisingly, this will usually fail whenever $p \ne n$.  While there is a natural morphism $\underline \Omega_X^{n-p} \to \underline \Omega_X^{n+p}$, these complexes are not connected by Grothendieck duality, a necessary piece of our proof of (\ref{equation symplectic composition dual}).  The correct generalization is to study a natural map $$L_\sigma^p: \mathbb D_X(\underline \Omega_X^{n+p}) \to \underline \Omega_X^{n+p},$$ where $\mathbb D_X(-) : = \mathbf R\mathscr Hom_{\mathscr O_X}(-, \omega_X^\bullet)[-2n]$ is the Grothendieck duality functor. Our main theorem considers the case when $p = n-1$:

\begin{theorem} \label{theorem main}
Let $X$ be a symplectic variety of dimension $2n$.  The morphism $$L_\sigma^{n-1}: \mathbb D_X(\underline \Omega_X^{2n-1}) \to \underline \Omega_X^{2n-1}$$ is a quasi-isomorphism.
\end{theorem}

\subsection{Vanishing theorems, higher Du Bois and higher rational singularities}

Since a symplectic variety $X$ satisfies the holomorphic extension property in all degrees, there is an isomorphism $$\pi_*\Omega_{\widetilde X}^{n-p} \cong \pi_*\Omega_{\widetilde X}^{n+p}$$ for every $0 \le p \le n$ (Proposition \ref{proposition reflexivity and symmetry of h^0's}).  To prove Theorem \ref{theorem main}, we study the commutative diagram \[ \begin{tikzcd}
    \mathbb D_X(\underline \Omega_X^{n+p}) \arrow{r} \arrow{d}{L_\sigma^p} & \mathbf R\mathscr Hom_{\mathscr O_X}(\pi_*\Omega_{\widetilde X}^{n+p}, \mathscr O_X) \arrow{d}{(\sigma^p)^*}\\ \underline \Omega_X^{n+p} \arrow{r} & \mathbf R\mathscr Hom_{\mathscr O_X}(\pi_*\Omega_{\widetilde X}^{n-p}, \mathscr O_X)
\end{tikzcd}.\] Since $X$ is symplectic, the vertical morphism $(\sigma^p)^*$ is an isomorphism.  By the Steenbrink vanishing theorem (\ref{equation steenbrink vanishing general}), we are able to deduce the induce maps $\mathscr H^j\mathbb D_X(\underline \Omega_X^{n+p}) \to \mathscr H^j\underline \Omega_X^{n+p}$ are injective for most $j$, and Theorem \ref{theorem main} follows as a special case of this analysis.  More generally, we are able to extend the local and Steenbrink vanishing theorems for symplectic varieties. 

\begin{theorem} \label{theorem local and steenbrink vanishing intro}

Let $X$ be a symplectic variety of dimension $2n$ with isolated singularities, and let $\pi:\widetilde X \to X$ a resolution of singularities.  

\begin{enumerate}
    \item For every $0 < p \le n$, we have $$R^j\pi_*\Omega_{\widetilde X}^{n-p}(\log E) = 0$$ for $n-p < j < n+p$.  Moreover, \begin{equation} \label{equation local vanishing degree 1 symplectic}
   R^j\pi_*\Omega_{\widetilde X}^1(\log E) = 0, \quad j > 1.\end{equation}

    \item For every $0 < p \le n$, we have $$R^j\pi_*\Omega_{\widetilde X}^{n+p}(\log E) = 0$$ for $n-p < j < n+p-1$.

    \item For every $0 < p \le n$, we have $$R^j\pi_*\Omega_{\widetilde X}^{n-p}(\log E)(-E) = 0$$ for $j \ge n-p$.  In particular, $$R^j\pi_*\Omega_{\widetilde X}^1(\log E)(-E) = 0, \quad j > 0.$$
\end{enumerate}   
\end{theorem}

The proof of local vanishing is an immediate consequence of Theorem \ref{theorem main}: under the restrictions on the singularities, there is a quasi-isomorphism $\mathbf R\pi_*\Omega_{\widetilde X}^{n-p}(\log E) \xrightarrow{\sim} \mathbb D_X(\underline \Omega_X^{n+p})$ (\S \ref{subsection du bois}).  The generalization of Steenbrink vanishing for isolated requires more work, but it depends heavily on the local vanishing theorem.

If $X$ has isolated singularities, our Steenbrink vanishing theorem can also be rephrased in terms of the Du Bois complex. Extending the notion established in \cite{mustata2021hodge} and \cite{jung2021higher} for hypersurfaces, we say that a symplectic variety $X$ is $k$-Du Bois if the natural map $$\Omega_X^{[p]} \to \underline \Omega_X^p$$ is a quasi-isomorphism for every $0 \le p \le k$, where $\Omega_X^{[p]}$ is the sheaf of reflexive $p$-forms\footnote{In general, there is no morphism $\Omega_X^{[p]} \to \underline \Omega_X^p$, but this is the case for rational singularities.  There is always a morphism $\Omega_X^p \to \underline \Omega_X^p$, where $\Omega_X^p$ is the sheaf of K\"ahler $p$-forms, and the $k$-Du Bois property for hypersurfaces requires these maps to be isomorphisms for $0 \le p \le k$.  However, terminal symplectic varieties are never lci \cite[Proposition 1.4]{beauville2000symplectic}, and $\Omega_X^p$ fails to be reflexive for even the nicest symplectic singularities, see also Proposition \ref{proposition reflexive kahler differentials}.}.  Therefore, symplectic germs are $1$-Du Bois.  There is also a dual generalization for rational singularities, extending the identification $\mathbb D_X(\mathbf R\pi_*\omega_{\widetilde X}) \cong \mathbf R\pi_*\mathscr O_{\widetilde X}$: we say that $X$ is $k$-rational if $$\Omega_X^{[p]} \to \mathbb D_X(\underline \Omega_X^{2n-p})$$ is a quasi-isomorphism for every $0 \le p \le k$.  The relationship between $k$-rational and $k$-Du Bois singularities is well-understood.  Given a variety $X$ with (0-)rational singularities, the $k$-rational property implies the $k$-Du Bois property \cite[Theorem B]{shen2023k}, and this relationship can be extended in special cases; for instance, using the description of the higher Du Bois property in terms of the minimal exponent, $k$-Du Bois hypersurfaces are $(k-1)$-rational \cite{mustata2021bois}, \cite{jung2021higher}.  For symplectic singularities, we get the following converse:

\begin{theorem} \label{theorem k db k rational}
Let $X$ be a symplectic variety of dimension $2n$ with isolated singularities.  If $X$ is $k$-rational, then $X$ is $(k+1)$-Du Bois; if $X$ is $n$-rational, then $X$ is $2n$-rational.
\end{theorem}

It would be interesting to know if symplectic singularities with isolated singularities are $(n-1)$-Du Bois, so that Theorem \ref{theorem k db k rational} is just a special case.  It is not true a symplectic variety is $k$-Du Bois for every $k$.  By Theorem \ref{theorem main}, $\underline \Omega_X^{2n-1}$ is a sheaf if and only if $\mathscr H^1\mathbb D_X(\underline \Omega_X^{2n-1}) = 0$.  Symplectic varieties of this kind of very special, and there is a connection between this vanishing and symplectic resolutions of singularities:

\begin{proposition}
Let $X$ be a terminal symplectic variety of dimension $2n$ satisfying $$\mathscr H^1\mathbb D_X(\underline \Omega_X^{2n-1}) = 0.$$  Then $X$ does not admit a proper symplectic resolution of singularities.
\end{proposition}

The converse is not true, as it may be the case that $X$ admits a \textit{local} symplectic resolution of singularities, or even a proper bimeromorphic model.

Coincidentally, the local vanishing (\ref{equation local vanishing degree 1 symplectic}) was also observed for toric varieties: in this case, the vanishing $R^1\pi_*\Omega_{\widetilde X}^1(\log E) = 0$ holds if and only if the toric variety is \textit{simplicial} \cite[Theorem 1.1]{shen2024local}.  It would be interesting to know necessary and sufficient conditions for (\ref{equation local vanishing degree 1 symplectic}) to hold for a rational singularity. 

\subsection{Applications to Hodge theory} The vanishing theorems and symplectic symmetry of the Du Bois complex indicate that symplectic singularities are very restrictive.  We consider some immediate applications regarding the Hodge theory of special symplectic varieties. First, we show that symplectic varieties with isolated singularities admitting a symplectic resolution are $(n-1)$-Du Bois, where $2n$ is the dimension.  We also show the cohomology groups of projective symplectic varieties with isolated singularities admitting a symplectic resolution all carry pure Hodge structures.  Second, we study the Hodge theory of projective $\mathbb Q$-factorial terminal primitive symplectic 4-folds (see \S\ref{subsection primitive symplectic varieties}) and show that the cohomology groups carry pure Hodge structures.  In this case, we show that $X$ is \textit{nearly} $1$-rational, and the cohomology groups agree with the intersection cohomology groups.  Using the decomposition theorem, we also describe the Hodge theory of the exceptional divisors in this case.

\subsection{On the intersection Hodge module} To conclude this paper, we discuss how Theorem \ref{theorem main} is connected to intersection cohomology.  In \cite[\S 3]{tighe2022llv}, the author constructs a symplectic symmetry \begin{equation} \label{equation intersection cohomology symmetry}
   IH^{n-p,q}(X) \xrightarrow{\sim} IH^{n+p,q}(X) 
\end{equation} for $0 \le p \le n$, where $IH^k(X, \mathbb C) = \bigoplus_{r + s = k} IH^{r,s}(X)$ is the Hodge decomposition of the intersection cohomology groups associated to a primitive symplectic variety with isolated singularities (\S \ref{subsection primitive symplectic varieties}).  The proof is global, as the symmetry is only constructed at the level of cohomology, although the main motivation was to use this to extend the LLV structure theorem to the singular setting.

A local version of (\ref{equation intersection cohomology symmetry}) can be obtained by consider the intersection cohomology Hodge module.  Briefly, a pure Hodge module generalizes the notion of a variation of pure Hodge structures in the presence of singularities.  The intersection Hodge modules $\mathrm{IC}_X$ comes with the data of a perverse sheaf $IC_X^\bullet$, which is just the intersection cohomology complex with rational coefficients; a $\mathscr D$-module structure $\mathcal{IC}_X$ (with respect to some embedding in a smooth manifold) coming with a ``good'' Hodge filtration $F$; and an identification $$IC_X^\bullet \otimes \mathbb C = \mathrm{DR}(\mathcal{IC}_X),$$ where $\mathrm{DR}$ is the de Rham functor.  By construction, $\mathrm{IC}_X$ is the unique pure Hodge module extending the variation of Hodge structure $\mathbb C_U[2n]$ on the regular locus.  Because of this, the symplectic form induces morphisms $$L_\sigma^p:\mathrm{gr}_{-(n-p)}^F\mathrm{DR}(\mathcal{IC}_X)[n+p] \to \mathrm{gr}_{-(n+p)}^F\mathrm{DR}(\mathcal{IC}_X)[n-p].$$ for each $p$.  A nearly identical proof to Theorem \ref{theorem main} gives the following:

\begin{theorem} \label{theorem intersection hodge module}
 Let $X$ be a symplectic variety of dimension $2n$. Then $$L_\sigma^{n-1}:\mathrm{gr}_{-1}^F\mathrm{DR}(\mathcal{IC}_X)[2n-1] \xrightarrow{\sim} \mathrm{gr}_{-(2n-1)}^F\mathrm{DR}(\mathcal{IC}_X)[1]$$ is a quasi-isomorphism.    
\end{theorem}

In particular, we get the symmetry \begin{equation}\label{equation intersection symmetry general intro}
IH^{1,1}(X) \xrightarrow{\sim} IH^{2n-1,1}(X)\end{equation} for any proper symplectic variety by taking hypercohomology.  We remark there is a connection between Theorem \ref{theorem main} and Theorem \ref{theorem intersection hodge module}, as the Du Bois complex underlies the holomorphic data of a complex of mixed Hodge module related to the cohomology of $X$.  Thus, the Hodge filtration associated to $\underline \Omega_X^\bullet$ induces the Hodge filtration on $\mathrm{IC}_X$; see \S \ref{section intersection module} for more details.  We also include a discussion on how the intersection cohomology has a deformation theoretic interpretation, using the symmetry (\ref{equation intersection symmetry general intro}). 

\subsection{Acknowledgements} I would like to thank Nicolas Addington, Mircea Musta\c t\u a, Mihnea Popa, Rosie Shen, Sridhar Venkatesh, Duc Vo for discussions and correspondences.  I was partially supported by NSF grant no. DMS-2039316.

\section{Preliminaries}

\subsection{Notation} If $X$ is a symplectic variety, we let $2n = \dim X$  If $\pi: \widetilde X \to X$ is a log-resolution, we set $$\Omega_{\widetilde X}^k(\log E)(-E) : = \Omega_{\widetilde X}^k(\log E) \otimes \mathscr I_E,$$ where $\mathscr I_E$ is the ideal sheaf of the exceptional divisor $E \subset \widetilde X$.

We write $H^k$ to be the cohomology of a complex, $\mathbb H^k$ to represent hypercohomology, and $\mathscr H^k$ for the cohomology sheaves.

Finally, we let $\omega_X^\bullet$ be the dualizing complex of $X$; if $X$ is Cohen-Macaulay, then $\omega_X^\bullet = \omega_X[\dim X]$.

\subsection{The Du Bois complex} \label{subsection du bois} Let $X$ be a complex algebraic variety.  The Du Bois complex $(\underline \Omega_X^\bullet, F)$ is an object in the derived category of filtered complexes of constructible sheaves, generalizing the holomorphic de Rham complex for algebraic varieties over $\mathbb C$.  We denote its graded pieces by $\underline \Omega_X^p : = \mathrm{gr}_F^p\underline \Omega_X^\bullet[p]$.  Studied by Du Bois \cite{du1981complexe} from Deligne's construction of the mixed Hodge structure \cite{deligne1974theorie}, the Du Bois complex is constructed by simplicial or cubical hyperresolutions of $X$.  We will not need this construction in this paper but will only use its formal consequences.  The interested reader can consult \cite{peters2008mixed} for a good treatment of this construction.

\begin{theorem} \label{theorem du Bois properties}
For $X$ a complex scheme of finite type and $\underline \Omega_X^{\bullet}$ its Du Bois complex, we have

\begin{enumerate}[label=\normalfont(\roman*)]
    \item \cite[4.5. Th\'eor\`eme]{du1981complexe} $\underline \Omega_X^{\bullet} \cong_{\mathrm{qis}} \mathbb C_X$.
    
    \item \cite[(3.2.1)]{du1981complexe}If $f:Y \to X$ is a proper morphism of finite type schemes, then there is a morphism $f^*: \underline \Omega_X^{\bullet} \to \mathbf Rf_*\underline \Omega_Y^{\bullet}$ in $D_{\mathrm{filt}}^b(X)$.
    
    \item\label{theorem du bois open restriction}\cite[3.10 Corollaire]{du1981complexe} If $U \subset X$ is an open subscheme then $\underline \Omega_X^{\bullet}|_U \cong_{\mathrm{qis}} \underline \Omega_U^{\bullet}$.
    
    \item \label{theorem du Bois properties natural map kahler} \cite[\S 3.2]{du1981complexe} There is a natural morphism $\Omega_X^{\bullet} \to \underline \Omega_X^{\bullet}$, where $\Omega_X^{\bullet}$ is the complex of K\"ahler differentials.  Moreover, this morphism is a quasi-isomorphism if $X$ is smooth.
    
    \item \label{theorem right triangle resolution exceptional divisor} \cite[4.11 Proposition]{du1981complexe} There is an exact triangle $$\underline \Omega_X^p \to \underline \Omega_{\Sigma}^p\oplus \mathbf R\pi_*\Omega_{\widetilde X}^p \to \mathbf R\pi_*\underline \Omega_E^p \xrightarrow{+1}$$ where $\pi:\widetilde X \to X$ is a morphism of complex schemes of finite type with singular locus $\Sigma \subset X$ and $E = \pi^{-1}(\Sigma)$. In particular, $\underline \Omega_X^n \cong_{\mathrm{qis}} \mathbf R\pi_*\omega_{\widetilde X} \cong_{\mathrm{qis}} \pi_*\omega_{\tilde X}$ if $\pi:\widetilde X \to X$ is a resolution of singularities. 
    
    \item \cite[4.5 Th\'eor\`eme]{du1981complexe} If $X$ is a proper variety, there is a spectral sequence $$E_1^{p,q} : = \mathbb H^q(X, \underline \Omega_X^p)\Rightarrow H^{p+q}(X, \mathbb C)$$ which degenerates at $E_1$ for every $p,q$.  Moreover, the filtration induced by this degeneration is equal to the Hodge filtration on the underlying mixed Hodge structure.

    \item\label{theorem right triangle log zero du Bois} \cite[\S 3.C]{kovacsdu2011} If $\pi:\tilde X \to X$ is a log-resolution of singularities with exceptional divisor $E$, and if $\Sigma$ is the singular locus, there is a right triangle $$\mathbf R\pi_*\Omega_{\tilde X}^p(\log E)(-E) \to \underline \Omega_X^p \to \underline \Omega_\Sigma^p \xrightarrow{+1}$$ which is independent of the choice of $\pi$.  

    \item \label{item vanishing for du bois complex} \cite[V.6.2]{guillen2006hyperresolutions} $\mathscr H^j\underline \Omega_X^k = 0$ for $j + k> \dim X$
\end{enumerate}
\end{theorem}

As a corollary of Theorem \ref{theorem du Bois properties}\ref{theorem right triangle log zero du Bois},\ref{item vanishing for du bois complex}, we get the \textit{Steenbrink vanishing theorem} \begin{equation}\label{equation steenbrink vanishing general}
    R^j\pi_*\Omega_{\widetilde X}^k(\log E)(-E) = 0, \quad j+k > \dim X,
\end{equation} see \cite[Theorem 2]{steenbrink1985vanishing}.

For any complex $K \in D_{\mathrm{coh}}^b(X)$, we define the Grothendieck duality functor \begin{equation}
    \mathbb D_X(K) : = \mathbf R\mathscr Hom_{\mathscr O_X}(K, \omega_X^\bullet)[-\dim X]
\end{equation} Of particular interest are the isomorphisms \begin{equation} \label{equation duality isomorphisms}
    \mathbb D_X(\mathbf R\pi_*\Omega_{\widetilde X}^p(\log E)(-E)) \cong_{\mathrm{qis}} \mathbf R\pi_*\Omega_{\widetilde X}^{\dim X-p}(\log E), \quad \mathbb D_X(\mathbf R\pi_*\Omega_{\widetilde X}^p) \cong_{\mathrm{qis}} \mathbf R\pi_*\Omega_{\widetilde X}^{\dim X-p}.
\end{equation} In particular, $\mathbb D_X(\underline \Omega_X^p) \xrightarrow{\sim} \mathbf R\pi_*\Omega_{\widetilde X}^{\dim X-p}(\log E)$ for $p > \dim \Sigma$ by Theorem \ref{theorem du Bois properties}\ref{theorem right triangle log zero du Bois}.

\subsection{Symplectic varieties}

\begin{definition}
A normal complex analytic variety $X$ is called a \textit{symplectic variety} if the regular locus $U$ admits a holomorphic symplectic form $\sigma \in H^0(U, \Omega_U^2)$ which extends holomorphically across any resolution of singularities.
\end{definition}

By \cite[Proposition 1.3]{beauville2000symplectic} (or \S \ref{subsection derived symmetries}), symplectic singularities are rational and therefore Gorenstein as the symplectic form defines a section $\sigma^n \in H^0(X, \omega_X)$, where $2n = \dim X$.  In particular, the dualizing complex satisfies $\omega_X^\bullet \cong \mathscr O_X[2n]$.  We will use the following version of Ischebeck's lemma, see \cite[\href{https://stacks.math.columbia.edu/tag/0A7U}{Tag 087U}]{stacks-project}:

\begin{lemma} \label{lemma dim support ischebeck}
If $X$ is a symplectic variety of dimension $2n$ and $\mathscr F$ a coherent $\mathscr O_X$-module, then $$\mathscr H^j\mathbb D_X(\mathscr F) = R^j\mathscr Hom_{\mathscr O_X}(\mathscr F, \mathscr O_X) = 0$$ for $j < 2n-\dim \mathrm{Supp}~\mathscr F.$ 
\end{lemma}

Since they are Gorenstein, symplectic varieties are canonical, and they are terminal if and only if $$\mathrm{codim}_X(\Sigma) \ge 4$$ \cite[Corollary 1]{namikawa2001note}. 
 Conversely, any complex variety with rational singularities which admits a symplectic form on its regular locus is necessarily a symplectic variety, as any differential form on $U$ extends across the singularities by \cite[Corollary 1.8]{kebekus2021extending} or \cite[Theorem 4]{namikawa2000extension}.  Therefore, if $\pi:\widetilde X \to X$ is a resolution of singularities, the inclusion \begin{equation} \label{equation holomorphic extension}
     \pi_*\Omega_{\widetilde X}^p \hookrightarrow \Omega_X^{[p]} : = j_*\Omega_U^p
 \end{equation} is an isomorphism for every $0 \le p \le 2n$.  Equivalently, $\pi_*\Omega_{\widetilde X}^p$ is reflexive for each $p$.  An immediate consequence is the symmetry \begin{equation} \label{equation symplectic symmetry GR differentials}
    \pi_*\Omega_{\widetilde X}^{n-p} \cong \pi_*\Omega_{\widetilde X}^{n+p}
\end{equation} induced by the isomorphism (\ref{equation symplectic symmetry regular locus}).

Next, we extend this symmetry to the sheaves $\mathscr H^0\underline \Omega_X^p$ and $\mathscr H^0\mathbb D_X(\underline \Omega_X^{2n-p})$.  By \cite{shen2023k}, there is an isomorphism $\mathscr H^0\mathbb D_X(\underline \Omega_X^{2n-p}) \xrightarrow{\sim}\pi_*\Omega_{\widetilde X}^p(\log E)$ for each $p$, which is reflexive by holomorphic extension.  Moreover, since $X$ has rational singularities, there is an isomorphism $\mathscr H^0\underline \Omega_X^p \xrightarrow{\sim} \Omega_X^{[p]}$, see \cite[p. 8]{tighe2023holomorphic}.  In particular:

\begin{proposition} \label{proposition reflexivity and symmetry of h^0's}
If $X$ is a symplectic variety of dimension $2n$, then there are isomorphisms $$\mathscr H^0\underline \Omega_X^{n-p} \cong \mathscr H^0\mathbb D_X(\underline \Omega_X^{n+p}) \xrightarrow{\sim} \mathscr H^0\underline \Omega_X^{n+p} \cong \mathscr H^0\mathbb D_X(\underline \Omega_X^{n-p})$$ induced by the symplectic form.
\end{proposition}

The reflexivity of the sheaves $\mathscr H^0\underline \Omega_X^k$ motivates the following definition:

\begin{definition} \label{definition higher du bois}
Let $X$ be a symplectic variety (or more generally a complex analytic variety with rational singularities).

\begin{enumerate}
    \item We say that $X$ is $k$-Du Bois if the natural morphism $$\Omega_X^{[p]} \to \underline \Omega_X^p$$ is a quasi-isomorphism for each $0 \le p \le k$.

    \item We say that $X$ is $k$-rational if the natural morphism $$\Omega_X^{[p]} \to \mathbb D_X(\underline \Omega_X^{2n-p})$$ is a quasi-isomorphism for each $0 \le p \le k$.
\end{enumerate}
\end{definition}

This is a slight weakening of the standard definitions introduced in \cite{jung2021higher} and \cite{friedmanlaza2024higher}, which requires the natural morphism of Theorem \ref{theorem du Bois properties}\ref{theorem du Bois properties natural map kahler} to be a quasi-isomorphism, and we emphasize that there is not a natural morphism $\Omega_X^{[p]} \to \underline \Omega_X^p$ for a general algebraic variety. It is worth noting that $\Omega_X^1$ fails to be reflexive, even for the simplest symplectic singularities (e.g., $(\mathbb C^{2n}/\langle \pm 1\rangle, 0)$).  It is unclear if $\Omega_X^1$ can ever be reflexive for symplectic varieties.  However, the following is evidence that K\"ahler forms are difficult to work with in this case:

\begin{proposition} \label{proposition reflexive kahler differentials}
Let $X$ be a symplectic variety of dimension $2n$.  Suppose the symplectic form $\sigma$ extends to a K\"ahler form (e.g., $\Omega_X^2$ is reflexive).

\begin{enumerate}
    \item If $\Omega_X^k$ is reflexive for $k < n$, then $\Omega_X^{2n-k}$ is cotorsion-free.

    \item If $\Omega_X^k$ and $\Omega_X^{2n-k}$ are reflexive, then $X$ is smooth.
\end{enumerate}
\end{proposition}

\begin{proof}
We say that a coherent sheaf $\mathscr F$ is cotorsion-free if the natural map $\mathscr F \to \mathscr F^{**}$ is surjective.  The restrictions on $(X,\sigma)$ immediately imply (1).  For (2), the assumptions imply $\Omega_X^k$ and $\Omega_X^{2n-k}$ are isomorphic; this means at each point $x \in X$, minimal generating sets of $\Omega_{X,x}^k$ and $\Omega_{X,x}^{2n-k}$ are equal.  The rank of these sets are $\binom{e}{k}$ and $\binom{e}{2n-k}$, respectively, where $e$ is the embedding dimension of $(X,x)$, see \cite[p.14]{graf2015generalized}.  This implies the embedding dimension must be $2n$.  
\end{proof}

In what follows, we will need the following property:

\begin{proposition} \label{proposition symstrat}
 Let $X$ be a symplectic variety.  
 
 \begin{enumerate}
     \item There is a stratification $X = X_0 \supset X_1 \supset X_2\supset ...$ of $X$ given by the singular locus of $X$, so that each $X_i = (X_{i-1})_{\mathrm{sing}}$.  The normalization of each $X_{i}$ is a symplectic variety, and $X_i^\circ : = (X_i)_{\mathrm{reg}}$ admits a global holomorphic symplectic form.  In particular, the singular locus $\Sigma$ has even dimension.
     \item Suppose that $x \in X$ such that $x \in X_{i}^\circ$.  Let $\widehat{X}_x$ and $\widehat{X_{i}^\circ}_x$ be the completion of $X$ and $X_{i}^\circ$ at $x$, respectively.  Then there is a decomposition $$\widehat{X}_x \cong Y_x\times \widehat{X_{i}^\circ}_x,$$ where $Y_x$ is a symplectic variety defined in an analytic neighborhood of $x$.
 \end{enumerate}   
\end{proposition}

\begin{proof}
This is \cite[Theorem 2.3]{kaledin2006symplectic}, except the symplectic variety $Y_x$ is a priori defined only on the formal completion (and so we cannot use \cite{artin1969algebraic}).  However, this is in fact the case and is discussed in \cite[Appendix A, Proposition 2.3]{kaplan2023crepant}.
\end{proof}

\begin{example} \label{example product decomposition}
If $X$ has rational singularities, $\pi_*\Omega_{\widetilde X}^k(\log E)(-E)$ is not reflexive except possibly for $k > \dim \Sigma$. If $X$ is a symplectic variety, we can say a bit more due to the local product decomposition.  Since the problem of reflexivity is local, we can shrink $X$ as necessary and assume $X \cong Y_x \times U_x$, where $U_x \cong \widehat{X_i^\circ}_x$.  Since $U_x$ is smooth, there is a decomposition $$\pi_*\Omega_{\widetilde X}^k(\log E)(-E) = \bigoplus_{r + s = k}(\pi_x)_*\Omega_{\widetilde{Y_x}}^r(\log E_x)(-E_x)\otimes \Omega_{U_x}^s,$$ where $\pi_x: \widetilde{Y_x} \to Y_x$ is a log-resolution of singularities with exceptional divisor $E_x$.  For example, $\pi_*\Omega_{\widetilde X}^k(\log E)(-E)$ is reflexive if $s > \dim (Y_x)_{\mathrm{sing}}$ for every $k \ge s \ge k-\dim U_x$.
\end{example}

\section{Local and Steenbrink Vanishing for Symplectic Varieties} \label{section vanishing} Let $X$ be a terminal symplectic variety with singular locus $\Sigma$. If $\mathrm{codim}_X(\Sigma) > 2n-2p$, then $\mathscr H^0\underline \Omega_X^{2p} \cong \pi_*\Omega_{\widetilde X}^{2p}(\log E)(-E)$ by Theorem \ref{theorem du Bois properties}\ref{theorem right triangle log zero du Bois}.  For this choice of $p$, there is a morphism \begin{equation}
    \Omega_{\widetilde X}^{n-p}(\log E) \to \Omega_{\widetilde X}^{n+p}(\log E)(-E)
\end{equation} obtained by wedging with the symplectic form.  We therefore have the following:

\begin{theorem}\label{theorem factorization wedge}
Let $X$ be a terminal symplectic variety of dimension $2n$ with holomorphic symplectic form $\sigma \in H^0(X, \Omega_X^{[2]})$ and singular locus $\Sigma$.  If $\mathrm{codim}_X(\Sigma) > 2n-2p$, there is a morphism $$L_\sigma^p: = \mathbf R\pi_*\widetilde \sigma^p: \mathbf R\pi_*\Omega_{\widetilde X}^{n-p}(\log E) \to \mathbf R\pi_*\Omega_{\widetilde X}^{n+p}(\log E)(-E)$$ for any log-resolution of singularities $\pi:\widetilde X \to X$, where $\widetilde \sigma$ is the unique extension of $\sigma$ to $H^0(\widetilde X, \Omega_{\widetilde X}^2)$.  Moreover, this map does not depend on the choice of $\pi$.
\end{theorem} 

More generally, the symplectic form induces wedging maps
\begin{equation}\label{equation wedging morphisms obvious}
\Omega_{\widetilde X}^p(\log E)(-E) \to \Omega_{\widetilde X}^{p+2}(\log E)(-E), \quad \Omega_{\widetilde X}^p(\log E) \to \Omega_{\widetilde X}^{p+2}(\log E).
\end{equation} for any $0 \le p \le n$. 

Now we construct the morphism on the level of the Du Bois complex.  Recall from Theorem \ref{theorem du Bois properties}\ref{theorem right triangle log zero du Bois} that there is a natural morphism $\mathbf R\pi_*\Omega_{\widetilde X}^p(\log E)(-E) \to \underline \Omega_X^p$ for each $p$.  Dualizing, this gives a morphism $\mathbb D_X(\underline \Omega_X^p) \to \mathbf R\pi_*\Omega_{\widetilde X}^{2n-p}(\log E)$.  The desired morphism is then obtained as \begin{equation} \label{equation natural symplectic symmetry morphism Du Bois complex}
    \mathbb D_X(\underline \Omega_X^{n+p}) \to \mathbf R\pi_*\Omega_{\widetilde X}^{n-p}(\log E) \xrightarrow{L_\sigma^p} \mathbf R\pi_*\Omega_{\widetilde X}^{n+p}(\log E)(-E) \to \underline \Omega_X^{n+p},
\end{equation} which we also call $L_\sigma^p$. We also have natural morphisms \begin{equation} \label{equation natural morphism from du bois to dual}
\underline \Omega_X^p \to \mathbb D_X(\underline \Omega_X^{2n-p})\end{equation} factoring through $\mathbf R\pi_*\Omega_{\widetilde X}^p$, and this morphism does not depend on the choice of resolution of singularities\footnote{To see this there is a map which only depends on $X$, one needs to pass through the intersection cohomology Hodge module; see \cite[\S 4.5]{saito1990mixed} or \S\ref{section intersection module}}.  We therefore obtain morphisms \begin{equation} \label{equation non isomorphic symplectic morphisms Du Bois}
 \underline \Omega_X^{n-p} \to \underline \Omega_X^{n+p}, \quad \mathbb D_X(\underline \Omega_X^{n+p}) \to \mathbb D_X(\underline \Omega_X^{n-p})   
\end{equation} factoring through (\ref{equation natural symplectic symmetry morphism Du Bois complex}).

\subsection{Proof of Main Theorem: isolated singularities} In this section, we prove Theorem \ref{theorem main} and Theorem \ref{theorem local and steenbrink vanishing intro} when $X$ has isolated singularities.  We first show the local vanishing theorem:

\begin{theorem} \label{lemma local vanishing isolated singularities}
Suppose $X$ is a singular symplectic variety of dimension $2n$ with isolated singularities. For every $0 < p \le n$, we have $$R^j\pi_*\Omega_{\widetilde X}^{n-p}(\log E) = 0, \quad n-p < j < n+p.$$  Moreover, $L_\sigma^{n-1}$ is a quasi-isomorphism.
\end{theorem}

\begin{proof}
We proceed in several steps: \newline 

\noindent \textbf{Step 1}. For any complex $K$ concentrated in non-negative degree, there is an exact triangle $$\mathscr H^0K \to K \to \tau_{\ge 1}K \xrightarrow{+1},$$ where $\tau_{\ge p}$ is the truncation of $K$ in cohomological degree $\ge p$.  By (\ref{equation duality isomorphisms}), there is a morphism of triangles \begin{equation} \label{equation commutative diagram} 
\begin{tikzcd}
    \mathbb D_X(\tau_{\ge 1}\mathbf R\pi_*\Omega_{\widetilde X}^{n+p}(\log E)(-E)) \arrow{d} \arrow{r}{(\tau_{\ge 1}L_\sigma^p)^*}~~~~~ & ~~~~~ \mathbb D_X(\tau_{\ge 1} \mathbf R\pi_*\Omega_{\widetilde X}^{n-p}(\log E)) \arrow{d} \\
    \mathbf R\pi_* \Omega_{\widetilde X}^{n-p}(\log E) \arrow{r}{(L_\sigma^p)^*} \arrow{d} & \mathbf R\pi_*\Omega_{\widetilde X}^{n+p}(\log E)(-E) \arrow{d}\\
    \mathbb D_X(\pi_*\Omega_{\widetilde X}^{n+p}) \arrow{r}{(\sigma^p)^*} & \mathbb D_X(\pi_*\Omega_{\widetilde X}^{n-p}),
    \end{tikzcd}
    \end{equation}obtained by dualizing this triangle and applying the morphism of Theorem \ref{theorem factorization wedge}, where $$\sigma^p: \pi_*\Omega_{\widetilde X}^{n-p} \xrightarrow{\sim} \pi_*\Omega_{\widetilde X}^{n+p}$$ is the isomorphism (\ref{equation symplectic symmetry GR differentials}).  

    Consider the sheaves $R^j\mathscr Hom_{\mathscr O_X}(\tau_{\ge 1} \mathbf R\pi_*\Omega_{\widetilde X}^{n+p}(\log E)(-E), \mathscr O_X)$.  There is a spectral sequence \begin{equation} \label{equation important spectral sequence}
    \begin{split}
       E_2^{k,l} &= R^k\mathscr Hom_{\mathscr O_X}(\mathscr H^{-l}\tau_{\ge 1}\mathbf R\pi_*\Omega_{\widetilde X}^{n+p}(\log E)(-E), \mathscr O_X) \\ & \Rightarrow R^{k +l}\mathscr Hom_{\mathscr O_X}(\tau_{\ge 1}\mathbf R\pi_*\Omega_{\widetilde X}^{n+p}(\log E)(-E), \mathscr O_X). 
    \end{split}
    \end{equation} Of course, the cohomology sheaves $\mathscr H^{-l}\tau_{\ge 1}\mathbf R\pi_*\Omega_{\widetilde X}^{n+p}(\log E)(-E)$ are zero for $l \ge 0$ and have zero-dimensional support for $l < 0$.  Therefore, $E_2^{k,l} = 0$ for any $k \ne 2n$ by Lemma \ref{lemma dim support ischebeck}, and $E_2^{k,l}$ degenerates on the second page.  It follows that \begin{equation}\label{equation important identity}
    \begin{split}
    R^j\mathscr Hom_{\mathscr O_X}(\tau_{\ge 1}\mathbf R\pi_*\Omega_{\widetilde X}^{n+p}&(\log E)(-E), \mathscr O_X)\\ & = R^{2n}\mathscr Hom_{\mathscr O_X}(R^{2n-j}\pi_*\Omega_{\widetilde X}^{n+p}(\log E)(-E), \mathscr O_X).
    \end{split}
    \end{equation} Returning to (\ref{equation commutative diagram}), we can see $R^j\pi_*\Omega_{\widetilde X}^{n-p}(\log E) \to R^j\pi_*\Omega_{\widetilde X}^{n+p}(\log E)(-E)$ is at least injective for $j < n+p$, using the fact that (\ref{equation important identity}) vanishes in this range by Steenbrink vanishing (\ref{equation steenbrink vanishing general}).  In particular, $R^j\pi_*\Omega_{\widetilde X}^{n-p}(\log E) = 0$ for $n-p < j < n+p$.   \newline
    
  \noindent \textbf{Step 2}. Let $\mathscr C_p$ be the mapping cone of $L_\sigma^p$.  For $j < n+p$, we get an exact sequence $$0 \to R^j\pi_*\Omega_{\widetilde X}^{n-p}(\log E) \to R^j\pi_*\Omega_{\widetilde X}^{n+p}(\log E)(-E) \to R^{j+1}\mathscr Hom(\mathscr C_p, \mathscr O_X) \to 0.$$ By Proposition \ref{proposition reflexivity and symmetry of h^0's} and Theorem \ref{theorem du Bois properties}\ref{theorem right triangle log zero du Bois}, $\mathscr H^0\mathscr C_p = 0$.  Moreover, the cohomology sheaves $\mathscr H^j\mathscr C_p$ have zero-dimensional support, so a similar application of the spectral sequence (\ref{equation important spectral sequence}) shows that $$R^{j + 1}\mathscr Hom_{\mathscr O_X}(\mathscr C_p, \mathscr O_X) = R^{2n}\mathscr Hom(\mathscr H^{2n-j-1}\mathscr C_p, \mathscr O_X).$$ But for $j < n+p$, $\mathscr H^{2n-j-1}\mathscr C_p = R^{2n-j}\pi_*\Omega_{\widetilde X}^{n-p}(\log E)$ by another application of Steenbrink vanishing. It follows that $L_\sigma^p$ is a quasi-isomorphism if and only if $$R^j\pi_*\Omega_{\widetilde X}^{n-p}(\log E) = 0$$ for $j \ge n+p$. \newline  

\noindent \textbf{Step 3}.  To conclude, we recall $R^{2n-1}\pi_*\Omega_{\widetilde X}^1(\log E) = 0$ by (\ref{equation local vanishing rational singularities}).  By Step 2, this implies $L_\sigma^{n-1}$ is a quasi-isomorphism.  
\end{proof}

\subsection{Steenbrink vanishing for symplectic varieties}

Next, we consider how Steenbrink vanishing extends for symplectic varieties.

\begin{theorem}\label{lemma steenbrink vanishing isolated singularities} Let $X$ be a  symplectic variety of dimension $2n$ with isolated singularities.  Then $$R^j\pi_*\Omega_{\widetilde X}^{n-p}(\log E)(-E) = 0$$ for $j \ge n-p$.  In particular, $X$ is $1$-Du Bois.
\end{theorem}     

\begin{proof}
By Steenbrink vanishing, we only need to check the vanishing for $n-p \le j \le n+p$.  We reduce to the edge cases $j = n-p$ or $n+p$ by studying the morphisms $R^j\pi_*\Omega_{\widetilde X}^{n-p}(\log E) \to R^j\pi_*\Omega_{\widetilde X}^{n+p}(\log E)$.  We do this by following an argument in \cite{friedman2022higher}. \newline 

\noindent \textbf{Step 1}.  Since the problem is local, we can shrink $X$ as necessary and prove the vanishing holds for cohomology groups. For each $k$, there is a natural morphism $$\mathbb H^j(\widetilde X, \Omega_{\widetilde X}^\bullet(\log E)(-E)/\mu^{\ge k+1}) \to \mathbb H^j(\widetilde X, \Omega_{\widetilde X}^\bullet(\log E)/\mu^{\ge k+1}),$$ where $\mu^\bullet$ is the stupid filtration, and this map is injective for each $j$ \cite[Lemma 3.1]{friedman2022higher}.  We consider two spectral sequences associated to these hypercohomology groups: \begin{equation} \label{equation spectral sequence truncated log -E}
\begin{split}
   (E_1^k)^{p,q} = H^q(\widetilde X, \mathscr H^p \mu_{\le k}\Omega_{\widetilde X}^\bullet(\log E)(-E))\Rightarrow \mathbb H^{p+q}(\widetilde X, \Omega_{\widetilde X}^\bullet(\log E)(-E)/\mu^{\ge k+1}) 
\end{split}
\end{equation} and \begin{equation} \label{equation spectral sequence log E}
    ('E_1^k)^{p,q} = H^q(\widetilde X, \mathscr H^p \mu_{\le k}\Omega_{\widetilde X}^\bullet(\log E)) \Rightarrow \mathbb H^{p+q}(\widetilde X, \Omega_{\widetilde X}^\bullet(\log E)/\mu^{\ge k+1}).
\end{equation}  If $k = 1$, note that we have shown that $('E_1^1)^{p,q} = 0$ for any $p,q > 0$.  Note also $H^q(\widetilde X, \mathscr O_{\widetilde X}(-E)) = 0$ for $q > 0$ by Theorem \ref{theorem du Bois properties}\ref{theorem right triangle log zero du Bois}, since $X$ has isolated Du Bois singularities. It follows that $$H^q(\widetilde X, \Omega_{\widetilde X}^1(\log E)(-E)) = (E_\infty^1)^{1,q} = ('E_\infty^1)^{1,q} = H^q(\widetilde X, \Omega_{\widetilde X}^1(\log E)) = 0$$ for $q > 1$, and $H^1(\widetilde X, \Omega_{\widetilde X}^1(\log E)(-E)) \hookrightarrow H^1(\widetilde X, \Omega_{\widetilde X}^1(\log E))$. \newline

\noindent \textbf{Step 2}. More generally, suppose we have shown that $H^j(\widetilde X,\Omega_{\widetilde X}^{n-p-1}(\log E)(-E)) = 0$ for $j \ge n-p-1$.  Another application of the spectral sequences $E_1^{n-p}$ and $'E_1^{n-p}$ shows $$H^j(\widetilde X,\Omega_{\widetilde X}^{n-p}(\log E)(-E)) \hookrightarrow H^j(\widetilde X, \Omega_{\widetilde X}^{n-p}(\log E))$$ for $j \ge n-p$, and  $H^j(\widetilde X, \Omega_{\widetilde X}^{n-p}(\log E)(-E)) = 0$ for $n-p < j < n+p$ by Theorem \ref{lemma local vanishing isolated singularities}. \newline 

\noindent \textbf{Step 3}. Next, we show that the vanishing also holds for the edge case $j = n-p$. To highlight the idea, we first show $$R^1\pi_*\Omega_{\widetilde X}^1(\log E)(-E) = 0.$$  For the complex manifold $\widetilde X$, the complex  $$0 \to i_!\mathbb C_{\widetilde X\setminus E} \to \mathscr O_{\widetilde X}(-E) \xrightarrow{d} \Omega_{\widetilde X}^1(\log E)(-E) \xrightarrow{d} \Omega_{\widetilde X}^2(\log E)(-E) \xrightarrow{d} ... \xrightarrow{d} \omega_{\widetilde X} \to 0$$ is exact, where $i: \widetilde X\setminus E \hookrightarrow \widetilde X$ is the inclusion.  We therefore get short exact sequences \begin{equation} \label{equation short exact sequence log zero sheaves}
       0 \to d\Omega_{\widetilde X}^{k-1}(\log E)(-E) \to \Omega_{\widetilde X}^k(\log E)(-E) \to d\Omega_{\widetilde X}^k(\log E)(-E) \to 0.
   \end{equation} By definition, the extended form $\widetilde \sigma$ is closed, and $\widetilde \sigma \wedge d\alpha = d(\widetilde \sigma \wedge \alpha)$ for any $k$-form $\alpha$, so (\ref{equation wedging morphisms obvious}) restricts to a map $$\widetilde \sigma: d\Omega_{\widetilde X}^k(\log E)(-E) \to d\Omega_{\widetilde X}^{k+2}(\log E)(-E).$$ Therefore, there is a commutative diagram \begin{equation}\label{equation diagram log zero sheaves and differentials} \begin{tikzcd} 
      R^j \pi_*(d\Omega_{\widetilde X}^{n-p -1}(\log E)(-E)) \arrow{r} \arrow{d}{\sigma^p} & R^j\pi_*\Omega_{\widetilde X}^{n-p}(\log E)(-E) \arrow{r} \arrow{d}{\sigma^p} & R^j\pi_*(d\Omega_{\widetilde X}^{n-p}(\log E)(-E)) \arrow{d}{\sigma^p} \\
      R^j\pi_*(d\Omega_{\widetilde X}^{n+p-1}(\log E)(-E)) \arrow{r} & R^j\pi_*\Omega_{\widetilde X}^{n+p}(\log E)(-E) \arrow{r} & R^j\pi_*(d\Omega_{\widetilde X}^{n+p}(\log E)(-E))
   \end{tikzcd}\end{equation} obtained by wedging with $\widetilde \sigma^p$.  In particular, there is a diagram \[ \begin{tikzcd}
      R^j \pi_*(d\mathscr O_{\widetilde X}(-E)) \arrow{r} \arrow{d}{\sigma^p} & R^j\pi_*\Omega_{\widetilde X}^{1}(\log E)(-E) \arrow{r} \arrow{d}{\sigma^p} & R^j\pi_*(d\Omega_{\widetilde X}^{1}(\log E)(-E)) \arrow{d}{\sigma^p} \\
      R^j\pi_*(d\Omega_{\widetilde X}^{2n-2}(\log E)(-E)) \arrow{r} & R^j\pi_*\Omega_{\widetilde X}^{2n-1}(\log E)(-E) \arrow{r} & R^j\pi_*(d\Omega_{\widetilde X}^{2n-1}(\log E)(-E))
   \end{tikzcd}\] By topological vanishing, the sheaves $R^j\pi_*(i_!\mathbb C_{\widetilde X\setminus E}) = 0$ for every $j$ \cite[Lemma 14.4]{greb2011differential}.  This implies that $R^j\pi_*(d\mathscr O_{\widetilde X}(-E)) = R^j\pi_*\mathscr O_{\widetilde X}(-E) = 0$ for $j > 0$, since the latter sheaves vanish by Theorem \ref{theorem du Bois properties}\ref{theorem right triangle log zero du Bois} and the fact that $X$ has at worst Du Bois singularities (or, our argument above).  Therefore, $R^j\pi_*\Omega_{\widetilde X}^1(\log E)(-E) \xrightarrow{\sim} R^j\pi_*(d\Omega_{\widetilde X}^1(\log E)(-E))$.  But see $R^1\pi_*(d\Omega_{\widetilde X}^{2n-1}(\log E)(-E)) = 0$ by Grauert-Riemenschneider vanishing, see Theorem \ref{theorem du Bois properties}\ref{theorem right triangle resolution exceptional divisor}. In particular, $$\mathrm{im}(R^1\pi_*\Omega_{\widetilde X}^1(\log E)(-E) \to R^1\pi_*\Omega_{\widetilde X}^{2n-1}(\log E)(-E)) = 0.$$ But the morphism $\mathbf R\pi_*\Omega_{\widetilde X}^1(\log E)(-E) \to \mathbf R\pi_*\Omega_{\widetilde X}^{2n-1}(\log E)(-E)$ factors through the quasi-isomorphism $L_\sigma^p: \mathbf R\pi_*\Omega_{\widetilde X}^{n-p}(\log E) \to \mathbf R\pi_*\Omega_{\widetilde X}^{n+p}(\log E)(-E)$.  In particular, $$R^1\pi_*\Omega_{\widetilde X}^1(\log E)(-E) \to R^1\pi_*\Omega_{\widetilde X}^{2n-1}(\log E)(-E)$$ is injective since $R^1\pi_*\Omega_{\widetilde X}^1(\log E)(-E) \to R^1\pi_*\Omega_{\widetilde X}^1(\log E)$ is injective.  This proves the claim. \newline 

  \noindent \textbf{Step 4}. Now we prove $R^{n-p}\pi_*\Omega_{\widetilde X}^{n-p}(\log E)(-E) = 0$ in general.  By (\ref{equation short exact sequence log zero sheaves}), we can assume that $$R^j\pi_*(d\Omega_{\widetilde X}^{n-p-1}(\log E)(-E)) = 0$$ for $j \ge n-p-1$ by induction.  By (\ref{equation diagram log zero sheaves and differentials}), it follows that \begin{equation*}
    \begin{split}
     \mathrm{im}(R^{n-p}&\pi_*\Omega_{\widetilde X}^{n-p}(\log E)(-E))\xrightarrow{\sigma^p} R^{n-p}\pi_*\Omega_{\widetilde X}^{n+p}(\log E)(-E)) \\& = \mathrm{im}(R^{n-p}\pi_*(d\Omega_{\widetilde X}^{n-p}(\log E)(-E)) \xrightarrow{\sigma^p} R^{n-p}\pi_*(d\Omega_{\widetilde X}^{n+p}(\log E)(-E))).   
    \end{split}
\end{equation*} An iterated application of Steenbrink vanishing shows $R^{n-p}\pi_*(d\Omega_{\widetilde X}^{n+p}(\log E)(-E)) = 0$, see \cite[Claim 14.7]{greb2011differential}, and so the image is zero.  But this morphism factors through $$R^j\pi_*\Omega_{\widetilde X}^{n-p}(\log E) \hookrightarrow R^j\pi_*\Omega_{\widetilde X}^{n+p}(\log E)(-E),$$ which is injective for $j < n+p$ by Theorem \ref{lemma local vanishing isolated singularities}, and again we see $R^{n-p}\pi_*\Omega_{\widetilde X}^{n+p}(\log E)(-E) = 0$. \newline 

\noindent \textbf{Step 5}. Now we prove $R^{n+p}\pi_*\Omega_{\widetilde X}^{n-p}(\log E)(-E) = 0$.  Note that $R^{2n-1}\pi_*\Omega_{\widetilde X}^1(\log E)(-E) = 0$ by (\ref{equation steenbrink vanishing extended}).  We also have exact sequences $$R^{2n-2}\pi_*\Omega_{\widetilde X}^1(\log E)(-E) \to R^{2n-2}\pi_*(d\Omega_{\widetilde X}^1(\log E)(-E)) \to R^{2n-1}\pi_*(d\mathscr O_{\widetilde X}(-E))$$ and $$R^{2n-2}\pi_*(d\Omega_{\widetilde X}^1(\log E)(-E)) \to R^{2n-2}\pi_*\Omega_{\widetilde X}^2(\log E)(-E) \to R^{2n-2}\pi_*(d\Omega_{\widetilde X}^2(\log E)(-E)).$$ We have seen the first and last terms of the first sequence vanish, and $R^{2n-2}\pi_*(d\Omega_{\widetilde X}^1(\log E)(-E)) = 0$.  This implies $R^{2n-2}\pi_*\Omega_{\widetilde X}^2(\log E)(-E) = 0$ since the last term of the second sequence vanishes, again by \cite[Claim 14.7]{greb2011differential}. Proceeding by induction, the claim will follow if we can show $$R^{n+p-1}\pi_*(d\Omega_{\widetilde X}^{n-p}(\log E)(-E)) = 0$$ for each $0 < p \le n$.  But this follows from $$R^{n+p-1}\pi_*\Omega_{\widetilde X}^{n-p}(\log E)(-E) \to R^{n+p-1}\pi_*(d\Omega_{\widetilde X}^{n-p}(\log E)(-E)) \to R^{n+p}\pi_*(d\Omega_{\widetilde X}^{n-p-1}(\log E)(-E)),$$ since the first term vanishes by assumption and the last term vanishes by an iterated application of Steenbrink vanishing.
\end{proof}

In addition, we get the following edge case for free:

\begin{corollary} \label{corollary extra vanishing}
If $X$ is a symplectic variety of dimension $2n$ with isolated singularities, then $$R^n\pi_*\Omega_{\widetilde X}^n(\log E)(-E) = 0.$$
\end{corollary}

\begin{proof}
    By shrinking $X$ as necessary, we only need to prove $H^n(\widetilde X, \Omega_{\widetilde X}^n(\log E)(-E)) = 0$.  In this case, there is a spectral sequence $$E_1^{p,q} = H^q(\widetilde X, \Omega_{\widetilde X}^p(\log(E)(-E)) \Rightarrow \mathbb H^{p + q}(\widetilde X, \Omega_{\widetilde X}^\bullet(\log E)(-E)) = 0.$$ By Theorem \ref{lemma steenbrink vanishing isolated singularities} and Steenbrink vanishing, it follows the differentials with source or target $$E_1^{n,n} = H^n(\widetilde X, \Omega_{\widetilde X}^n(\log E)(-E))$$ are 0, and $$H^n(\widetilde X, \Omega_{\widetilde X}^n(\log E)(-E)) = E_\infty^{n,n} = 0.$$
\end{proof}

\begin{remark} \label{remark zero map}
In the proof of Theorem \ref{lemma steenbrink vanishing isolated singularities}, we used the fact that the maps $$R^{j}\pi_*\Omega_{\widetilde X}^{j}(\log E)(-E) \to R^{j}\pi_*\Omega_{\widetilde X}^{n+p}(\log E)(-E)$$ are the zero morphisms for $j = n-p$.  This actually extends if we assume the $k$-Du Bois property.  Suppose that $X$ is $(n-p-1)$-Du Bois. One checks using the proof of Theorem \ref{lemma steenbrink vanishing isolated singularities} that  \begin{equation*}
    \begin{split}
     \mathrm{im}(R^{j}\pi_*\Omega_{\widetilde X}^{n-p}(\log E)&(-E)) \xrightarrow{\sigma^p} R^{j}\pi_*\Omega_{\widetilde X}^{n-p}(\log E)(-E)) \\& = \mathrm{im}(R^{j}\pi_*(d\Omega_{\widetilde X}^{n+p}(\log E)(-E)) \xrightarrow{\sigma^p} R^{j}\pi_*(d\Omega_{\widetilde X}^{n+p}(\log E)(-E))).   
    \end{split}
\end{equation*} for $j > 0$. The sheaves $R^j\pi_*(d\Omega_{\widetilde X}^{n+p}(\log E)(-E))$ don't need to vanish, but we claim that the morphism is zero anyway. Indeed, in every neighborhood of a point $x \in X$, a section $\alpha_{n+p}$ of $\Omega_{\widetilde X}^{n+p}(\log E)(-E)$ can be written as $\widetilde \sigma^p \wedge \alpha_{n-p}$ for some local section $\alpha_{n-p}$ of $\Omega_{\widetilde X}^{n-p}(\log E)(-E)$.  This follows since $\alpha_{n+p}$ is a unique extension of a form on the regular locus $U$, which is symplectic. 
 By taking cohomology classes, it follows that the image of second morphism is the same as the image $$\mathrm{im}(R^j\pi_*(d\Omega_{\widetilde X}^{n-p-2}(\log E)(-E)) \xrightarrow{\sigma^{p+1}} R^j\pi_*(d\Omega_{\widetilde X}^{n+p}(\log E)(-E))),$$ and the sheaves $R^j\pi_*(d\Omega_{\widetilde X}^{n-p-2}(\log E)(-E))$ vanish by the higher Du Bois property.

The argument used in Theorem \ref{lemma steenbrink vanishing isolated singularities} breaks down quickly if we try and prove a symplectic variety is $k$-Du Bois for $k > 1$.  Indeed, the spectral sequence $'E_1^{n-k}$ has too many non-zero terms in low degree, and we cannot confirm that the maps $R^j\pi_*\Omega_{\widetilde X}^{n-p}(\log E)(-E) \to R^j\pi_*\Omega_{\widetilde X}^{n-p}(\log E)$ are injective.  But it turns out that assuming $R^j\pi_*\Omega_{\widetilde X}^{n-p}(\log E) = 0$ for $j > 0$ gives the desired vanishing in the next degree.  
\end{remark}

\begin{theorem}
If $X$ is a symplectic variety with isolated singularities with $k$-rational singularities, then $X$ is $(k+1)$-Du Bois.  If $X$ is $n$-rational, then it is $2n$-rational.
\end{theorem}

\begin{proof}
Assume $X$ is $k$-rational.  By \cite[Theorem 3.2]{friedman2022higher}, a $k$-rational singularity is $k$-Du Bois\footnote{Specifically, they prove that the vanishing $\mathscr H^j\mathbb D_X(\underline \Omega_X^{\dim X-p}) = 0$ for $j > 0$ implies $\mathscr H^j\underline \Omega_X^p = 0$, which is enough for our definition.}.  Proceeding by induction, Remark \ref{remark zero map} implies that the morphisms $$R^j\pi_*\Omega_{\widetilde X}^{k+1}(\log E)(-E) \to R^j\pi_*\Omega_{\widetilde X}^{2n-k-1}(\log E)(-E)$$ are the zero morphisms for $j > 0$.  The assumption $R^j\pi_*\Omega_{\widetilde X}^p(\log E) = 0$ for $p \le k$ and $j > 0$ implies the morphisms $$R^j\pi_*\Omega_{\widetilde X}^{k+1}(\log E)(-E) \to R^j\pi_*\Omega_{\widetilde X}^{k+1}(\log E)$$ are injective by considering again the spectral sequences $E_1^{k+1}$ and $'E_1^{k+1}$ of Theorem \ref{lemma steenbrink vanishing isolated singularities}.  Therefore, $R^j\pi_*\Omega_{\widetilde X}^{k+1}(\log E)(-E) = 0$ for $j > 0$ by Theorem \ref{lemma local vanishing isolated singularities}.

For the second part, Theorem \ref{lemma local vanishing isolated singularities} and the first part imply $\underline \Omega_X^{k}$ is a sheaf for $k \ne n$.  But a similar application of the spectral sequence in Corollary \ref{corollary extra vanishing} implies $\underline \Omega_X^n$ is a sheaf, too (in particular, $X$ is $2n$-Du Bois).  By considering the triangles in (\ref{equation commutative diagram}), this implies $\mathbf R\mathscr Hom_{\mathscr O_X}(\pi_*\Omega_{\widetilde X}^k, \mathscr O_X)$ is a sheaf for each $k$.  The vanishing $R^j\pi_*\Omega_{\widetilde X}^{n+p}(\log E) = 0$ for $j > 0$ follows from (\ref{equation commutative diagram log E to log E}).  
\end{proof}

It seems that symplectic varieties with isolated singularities should be $k$-Du Bois for each $k < n$, although it is not true that $X$ is $k$-Du Bois for all $k$.  This is because symplectic varieties are not generally $k$-rational for $k < n$, as in the case $X$ admits a proper symplectic resolution of singularities (\S \ref{section symplectic resolutions}).  We do show $X$ is $(n-1)$-Du Bois for a symplectic germ $(X,x)$ admitting a proper symplectic resolution of singularities (Theorem \ref{theorem symplectic resolution k du bois}), where $\mathscr H^j\mathbb D_X(\underline \Omega_X^{n+p}) \ne 0$ if and only if $j = n-p$; if this holds for an arbitrary symplectic variety, then the desired injectivity of Remark \ref{remark zero map} holds.

\begin{corollary} \label{corollary vanishing higher log E}
If $X$ is a symplectic variety of dimension $2n$ with isolated singularities, then $$R^j\pi_*\Omega_{\widetilde X}^{n+p}(\log E) = 0$$ for each $n-p < j < n+p-1$.
\end{corollary}

\begin{proof}
There is an analogous diagram to (\ref{equation commutative diagram}) \begin{equation} \label{equation commutative diagram log E to log E} 
\begin{tikzcd}
    \mathbb D_X(\tau_{\ge 1}\mathbf R\pi_*\Omega_{\widetilde X}^{n+p}(\log E)(-E)\arrow{d} \arrow{r}{(\tau_{\ge 1}\mathbf R\pi_*\widetilde{\sigma}^p)^*}~~~~~ & ~~~~~ \mathbb D_X(\tau_{\ge 1} \mathbf R\pi_*\Omega_{\widetilde X}^{n-p}(\log E)(-E)) \arrow{d} \\
    \mathbf R\pi_* \Omega_{\widetilde X}^{n-p}(\log E) \arrow{r}{(\mathbf R\pi_*\widetilde \sigma^p)^*} \arrow{d} & \mathbf R\pi_*\Omega_{\widetilde X}^{n+p}(\log E) \arrow{d}\\
    \mathbb D_X(\pi_*\Omega_{\widetilde X}^{n+p}) \arrow{r}{(L_\sigma^p)^*} & \mathbb D_X(\pi_*\Omega_{\widetilde X}^{n-p}),
    \end{tikzcd}
    \end{equation}
By Theorem \ref{lemma steenbrink vanishing isolated singularities}, Steenbrink vanishing, and the (analogous) spectral sequence (\ref{equation important spectral sequence}), the cohomology sheaves of the duals of the truncated complexes vanish in degree $n+p >j > n-p$.  Thus, $R^j\pi_*\Omega_{\widetilde X}^{n-p}(\log E) \to R^j\pi_*\Omega_{\widetilde X}^{n+p}(\log E)$ is an isomorphism for $j < n-p-1$, and the claim follows from Theorem \ref{lemma local vanishing isolated singularities}.
\end{proof}

\subsection{Proof of Main Theorem: general case}\label{subsection general case} We now aim to prove Theorem \ref{theorem main} for a general symplectic variety.  The idea is to use the local product decomposition and reduce the problem down to a set whose support has dimension zero.  

\begin{theorem} \label{theorem local vanishing general}
Let $X$ be a symplectic variety.  Then $$R^j\pi_*\Omega_{\widetilde X}^1(\log E) = 0$$ for $j > 1$.
\end{theorem}

\begin{proof}
Note that we have shown the theorem when $X$ has isolated singularities; in particular, the theorem holds for surfaces\footnote{This is obvious since symplectic surface singularities are equivalent to ADE surface singularities; in this case, the cohomology sheaves $R^j\pi_*\Omega_{\widetilde X}^k(\log E)$ all vanish for $j > 0$}.  The problem is local, so by Proposition \ref{proposition symstrat}, we can assume around each singular point $x$ that $X \cong Y_x \times U_x$, where $Y_x$ is a symplectic variety and $U_x$ is smooth.  An application of the K\"unneth formula implies \begin{equation} \label{equation product decomposition for local vanishing}
    \begin{split}
        R^j\pi_*\Omega_{\widetilde X}^k(\log E) & \cong R^j(\pi_x\times \mathrm{id})_*\Omega_{\widetilde Y_x\times U_x}^k(\log E) \\ & \cong \bigoplus_{r+s = k} R^j(\pi_x\times \mathrm{id})_*(\Omega_{\widetilde Y_x}^r(\log E_x)\otimes \Omega_{U_x}^s)\\ &\cong \bigoplus_{r + s = k} R^j(\pi_x)_*\Omega_{\widetilde Y_x}^r(\log E_x)
    \end{split}
    \end{equation} where again $\pi_x:\widetilde{Y_x} \to Y_x$ is a log-resolution of singularities with exceptional divisor $E_x$.  Note that $R^j(\pi_x)_*\mathscr O_{\widetilde Y_x} = 0$ for $j > 0$.  By induction, we can assume $R^j(\pi_x)_*\Omega_{\widetilde{Y_x}}^{1}(\log E_x) = 0$ for $j > 1$ as long as $\dim Y_x < \dim X$.  Note that $\dim Y_x = \dim X$ if and only if $\dim U_x \le 0$, whence we can assume the vanishing $R^j\pi_*\Omega_{\widetilde X}^{1}(\log E) = 0$ for $j > 1$ holds outside a set $\Sigma_0$ with $\dim \Sigma_0 \le 0$.

    Now the proof follows the structure of Theorem \ref{lemma local vanishing isolated singularities}. Returning to the spectral sequence (\ref{equation important spectral sequence}), we can see the terms $$E_2^{k.l} = R^k\mathscr Hom_{\mathscr O_X}(R^{-l}\tau_{\ge 1}\mathbf R\pi_*\Omega_{\widetilde X}^{2n-1}(\log E)(-E), \mathscr O_X) = 0$$ whenever $k + l = 1$.  Indeed, the only possible non-zero terms by Steenbrink vanishing are of the form $(k,l) = (2,-1)$.  But by assumption, $\mathrm{codim}_X(\Sigma) \ge 4$, and so $$R^2\mathscr Hom_{\mathscr O_X}(R^1\pi_*\Omega_{\widetilde X}^{2n-1}(\log E)(-E), \mathscr O_X) = 0.$$  A similar argument shows $R^1\mathscr Hom_{\mathscr O_X}(\tau_{\ge 1}\mathbf R\pi_*\Omega_{\widetilde X}^1(\log E), \mathscr O_X) = 0$, since we can assume $R^j\pi_*\Omega_{\widetilde X}^1(\log E)$ has zero dimensional support.  Thus, $R^1\pi_*\Omega_{\widetilde X}^1(\log E) \xrightarrow{\sim} R^1\pi_*\Omega_{\widetilde X}^{2n-1}(\log E)(-E)$. The proof now follows as in Theorem \ref{lemma local vanishing isolated singularities}, Step 2.
\end{proof}

We would like to extend Theorem \ref{theorem main} for arbitrary symplectic singularities in any degree, but this may not be true for dimension reasons, even with the local product structure.  Indeed, it may be the case in (\ref{equation product decomposition for local vanishing}) that we encounter a factor $R^j(\pi_x)_*\Omega_{\widetilde Y_x}^r(\log E_x)$, where $r > \frac{1}{2} \dim Y_x$.  For isolated singularities, we do not claim in Corollary \ref{corollary vanishing higher log E} that $R^j\pi_*\Omega_{\widetilde X}^{n+p}(\log E) = 0$ for $j > n-p$, and in fact this should not be the case.  As an example, suppose $R^j\pi_*\Omega_{\widetilde X}^{2n-1}(\log E) = 0$ for $j > 1$.  Since $\tau_{\ge 1}\mathbf R\pi_*\Omega_{\widetilde X}^{1}(\log E)(-E)$ is the trivial complex by Theorem \ref{lemma steenbrink vanishing isolated singularities}, we have a quasi-isomorphism $\mathbf R\pi_*\Omega_{\widetilde X}^{2n-1}(\log E) \xrightarrow{\sim} \mathbf R\mathscr Hom_{\mathscr O_X}(\pi_*\Omega_{\widetilde X}^1, \mathscr O_X)$. The proof of Theorem \ref{lemma local vanishing isolated singularities} shows $$R^j\pi_*\Omega_{\widetilde X}^1(\log E) \xrightarrow{\sim} R^j\mathscr Hom_{\mathscr O_X}(\pi_*\Omega_{\widetilde X}^{2n-1}, \mathscr O_X)$$ for $j = 0,1$, so the assumption would imply that $$\mathbf R\pi_*\Omega_{\widetilde X}^1(\log E) \cong_{\mathrm{qis}} \mathbf R\mathscr Hom_{\mathscr O_X}(\pi_*\Omega_{\widetilde X}^{2n-1}, \mathscr O_X),$$ and by duality this implies $\mathbf R\pi_*\Omega_{\widetilde X}^{2n-1}(\log E)(-E)$ is isomorphic to a sheaf.  This typically does not happen, see \S \ref{section symplectic resolutions}.

As a corollary of the Proof of Theorem \ref{lemma local vanishing isolated singularities}, here is a case where Theorem \ref{theorem main} can be extended.

\begin{proposition}
Let $X$ be a terminal symplectic variety of dimension $2n$.  If $\Omega_X^{[k]} \cong_{\mathrm{qis}} \underline \Omega_X^k$ for some $k > n$, then $\Omega_X^{[2n-k]} \cong_{\mathrm{qis}} \mathbb D_X(\underline \Omega_X^k)$.
\end{proposition}

\begin{proof}
The proof is a generalization of the argument we give in \S\ref{subsection derived symmetries}.  By assumption, there is a composition $$\Omega_X^{[2n-k]} \to \mathbb D_X(\underline \Omega_X^{k}) \to \underline \Omega_X^k \xrightarrow{\sim} \Omega_X^{[k]}$$ which is an isomorphism.  The result follows by applying $\mathbb D_X(-)$.  
\end{proof}

\section{Hodge Theory of Crepant Morphisms} \label{section symplectic resolutions} Let $X$ be a symplectic variety.  A crepant morphism $\phi:Z \to X$ is a proper birational morphism which is an isomorphism outside $\Sigma$ and satisfies $\phi^*K_X = K_Z$.  In this case, $Z$ is also a symplectic variety.  An important case of crepant morphisms are \textit{symplectic resolutions}, where $Z$ is a holomorphic symplectic manifold.

Symplectic resolutions are an important class of morphisms since they provide new classes of compact hyperk\"ahler manifolds, and their geometry has been extensively studied.  By  \cite[Lemma 2.11]{kaledin2006symplectic}, symplectic resolutions are semismall, and the fibers $L_x = \phi^{-1}(x)$ have Hodge-Tate cohomology: this means that $H^k(L_x, \mathbb C)$ is pure and only supported in Hodge degree $(p,p)$.  In particular, $H^k(L_x, \mathbb C) = 0$ for odd $k$.  The idea of the latter uses the fact that $R^j\phi_*\Omega_Z^p = 0$ for $j + p > 2n$.  We generalize this slightly for crepant morphisms:

\begin{proposition}
If $\phi:Z \to X$ is a crepant morphism, then $$R^j\phi_*\underline \Omega_Z^{n+p} = 0$$ for $j > n-p$ and $0 < p \le n$.
\end{proposition}
\begin{proof}
By \cite[Proposition 2.16]{tighe2022llv}, any crepant morphism of symplectic varieties is semismall.  In particular, the fibers $L_x = \phi^{-1}(x)$ satisfy $\dim L_x \le n$. Note for $j > n-p$, Theorem \ref{theorem du Bois properties}\ref{theorem right triangle resolution exceptional divisor} and Steenbrink vanishing implies $$R^j\phi_*\underline \Omega_Z^{n+p} \xrightarrow{\sim} R^j\phi_*\underline \Omega_L^{n+p},$$ and it is enough to check the vanishing on stalks. Consider the hypercohomology spectral sequence $$E_1^{k,l} = R^k\phi_*\mathscr H^l\underline \Omega_L^{n+p} \Rightarrow R^{k + l}\phi_*\underline \Omega_L^{n+p}.$$  By cohomology and base change, we have $$\widehat{R^k\phi_*\mathscr H^l\underline\Omega_L^{n+p}}_x = H^k(L_x, \widehat{\mathscr H^l\underline \Omega_L^{n+p}}_x).$$ Theorem \ref{theorem du Bois properties}\ref{theorem du bois open restriction} implies that $\widehat{\mathscr H^l\underline \Omega_L^{n+p}}_x$ is isomorphic to $\mathscr H^l\underline \Omega_{L_x}^{n+p}$, which vanishes if $p > 0$.  Therefore, $R^j\phi_*\underline \Omega_Z^{n+p} = 0$ for $j > n-p$.  
\end{proof}

If $\phi:Z \to X$ is a symplectic resolution of singularities, then $\dim L \le 2n-2$.  It follows from Theorem \ref{theorem du Bois properties}\ref{theorem right triangle log zero du Bois} that $$\underline \Omega_X^{2n-1} \xrightarrow{\sim} \mathbf R\phi_*\Omega_Z^{2n-1},$$ and so $\mathbb D_X(\underline \Omega_X^{2n-1}) \cong \mathbf R\phi_*\Omega_Z^1$ by duality (this gives a simple proof of Theorem \ref{theorem main} in this case).  We therefore obtain the following:

\begin{corollary} \label{corollary vanishing symplectic resolution}
Suppose $X$ is a symplectic variety.  If $\mathscr H^1\mathbb D_X(\underline \Omega_X^{2n-1}) = 0$ (equivalently, if $X$ is 1-rational), then $X$ cannot admit a proper symplectic resolution of singularities.
\end{corollary}

The converse is not true, since the same proof shows $\mathbf R\phi_*\mathbb D_Z(\underline \Omega_Z^{2n-1}) \cong \mathbb D_X(\underline \Omega_X^{2n-1})$ for any crepant morphism $\phi: Z \to X$.  This suggests the birational geometry of 1-rational symplectic varieties is quite tame.

\begin{example}
The preceding observation can be taken further if we restrict $\dim \Sigma$.  For instance, suppose that $X$ only has isolated singularities and $\phi: Z \to X$ is a symplectic resolution.  For every $0 < p \le n$, we have $$\underline \Omega_X^{n+p} \xrightarrow{\sim} \mathbf R\phi_*\Omega_Z^{n+p}, \quad \mathbb D_X(\underline \Omega_X^{n+p}) \xrightarrow{\sim} \mathbf R\phi_*\Omega_Z^{n-p}$$ (this gives a simple proof of Theorem \ref{theorem main} in this case).  It follows that the complexes $\mathbf R\phi_*\Omega_Z^k$ are independent of the choice of symplectic resolution, except possibly when $k = n$. For this case, note that the sheaves $R^j\phi_*\Omega_Z^n$ are independent of $\phi$, except possibly for $j = n$, by Theorem \ref{theorem du Bois properties}\ref{theorem right triangle log zero du Bois} and the vanishing $R^j\phi_*\underline \Omega_L^n = 0$ for $j \ne n$.  In fact, by Corollary \ref{corollary extra vanishing}, $$R^n\phi_*\Omega_Z^n \xrightarrow{\sim} R^n\phi_*\underline \Omega_L^n.$$  We see $H^n(L_x, \underline \Omega_L^n) = H^{2n}(L_x, \mathbb C)$, and $(R^n\phi_*\Omega_Z^n)_x$ depends only on $\dim H^n(L_x, \underline \Omega_L^n)$.  This suggests that $\mathbf R\phi_*\underline \Omega_Z^n$ is also independent of $\phi$ if two symplectic resolutions are deformation equivalent, which is conjectured to holds for any two symplectic resolutions of $X$.  In the projective case, this is Huybrechts' theorem \cite[Corollary 4.7]{huybrechts1999basic}. 
\end{example}

Finally, we study the $k$-Du Bois property for symplectic varieties admitting a symplectic resolution:

\begin{theorem} \label{theorem symplectic resolution k du bois}
If $X$ is a symplectic variety of dimension $2n$ with isolated singularities admitting a symplectic resolution $\phi: Z \to X$, then $X$ is $(n-1)$-Du Bois.  Moreover, if $X$ is projective, then $H^k(X, \mathbb Q)$ carries a pure Hodge structure for each $k$.
\end{theorem} 

\begin{proof}
The second part is easy to see: there is a long exact sequence $$... \to H^k(X, \mathbb Q) \to H^k(Z, \mathbb Q) \to H^k(L, \mathbb Q) \to ...$$ for each $k$, noting that all the spaces involved are assumed projective.  The result follows since $H^k(L, \mathbb Q)$ is Hodge-Tate.

For the first part, we recall from Remark \ref{remark zero map} that it is enough to show that $\mathscr H^j\underline \Omega_X^{n-p} \hookrightarrow \mathbb D_X(\underline \Omega_X^{n+p}) \cong \mathbf R\phi_*\Omega_Z^{n-p}$ for $j > 0$ and $0 < p \le n$.  This follows again from the fact that $R^j\phi_*\underline \Omega_L^{n-p} = 0$ for $j \ne n-p$, since $\mathscr H^j\underline \Omega_X^{n-p} = 0$ for $j > n-p$ by Theorem \ref{theorem local and steenbrink vanishing intro}.
\end{proof}

We can also see that symplectic varieties admitting symplectic resolutions with a smooth singular locus is also $(n-1)$-Du Bois by Proposition \ref{proposition symstrat}.  In this case, each transversal slice $Y_x$ has isolated singularities, and a crepant resolution on $X$ descends to a crepant resolution on $Y_x$ by functorial pull-back \cite[Theorem 1.11]{kebekus2021extending}.

\section{Cohomology of Primitive Symplectic 4-Folds} As another application, we study the Hodge theory of $\mathbb Q$-factorial terminal primitive symplectic 4-folds. If $\pi:\widetilde X \to X$ is a log-resolution of singularities with exceptional divisor $E$ and smooth components $E_i$, we let $E_{(m)}$ be the disjoint union of all $m$-fold intersections of the $E_i$.

\subsection{Hodge theory of primitive symplectic varieties} \label{subsection primitive symplectic varieties}

\begin{definition}
A projective variety (or more generally, a compact K\"ahler variety) $X$ is a \textit{primitive symplectic variety} if $X$ is a symplectic variety with symplectic form $\sigma$ satisfying the following additional properties:

\begin{enumerate}
    \item $H^1(X, \mathscr O_X) = 0$; 

    \item $H^0(X, \Omega_X^{[2]}) = \langle \sigma \rangle$.
\end{enumerate}
\end{definition}

Primitive symplectic varieties are singular analogues of irreducible holomorphic symplectic manifolds; in fact, a primitive symplectic variety is smooth if and only if it is irreducible holomorphic symplectic \cite[Theorem 1]{schwald2022definition}.

 When $X$ is primitive symplectic, there is a natural inclusion $H^2(X, \mathbb Z) \hookrightarrow H^2(\widetilde X, \mathbb Z)$ for any resolution of singularities, and the Hodge filtration agrees with the degeneration of a spectral sequence  \begin{equation}\label{equation hodge to de rham X}
E_1^{p,q} = H^q(X, \pi_*\Omega_{\widetilde X}^p) \Rightarrow \mathbb H^{p+q}(X, \pi_*\Omega_{\widetilde X}^\bullet)\end{equation} for $p + q \le 2$ \cite[Lemma 2.2]{bakker2016global}.

Let $IH^k(X, \mathbb Q)$ be the $k^{th}$-intersection cohomology group.  If $X$ is is primitive symplectic, then $X$ carries a pure Hodge structure of weight $k$ by an inclusion $IH^k(X, \mathbb Q) \hookrightarrow H^k(\widetilde X, \mathbb Q)$ induced by the decomposition theorem.  If $X$ has isolated singularities, then $IH^k(X, \mathbb Q) \cong H^k(U, \mathbb Q)$ for $k < \dim X$.  In this case, the Hodge filtration is the one induced by the spectral sequence \begin{equation}\label{equation spectral sequence log regular locus}
    E_1^{p,q} = H^q(\widetilde X, \Omega_{\widetilde X}^p(\log E)) \Rightarrow H^{p+q}(\widetilde X, \Omega_{\widetilde X}^\bullet(\log E)) \cong H^{p+q}(U, \mathbb C),
\end{equation} where $\pi:\widetilde X \to X$ is a log-resolution of singularities.  By duality, this gives the Hodge filtration on $IH^k(X, \mathbb C) \cong H^k_c(U, \mathbb C)$ for $k > \dim X$.  Since $X$ has isolated singularities, we also have that the Hodge-to-de Rham spectral sequence 
\begin{equation}\label{equation hodge to de rham regular locus}
 E_1^{p,q} = H^q(U, \Omega_U^p) \Rightarrow H^{p+q}(U, \mathbb C) 
\end{equation} degenerates in the range $p + q < \dim X - 1$ by \cite[Thm. 2]{arapura1990local}\footnote{In fact, this was upgraded to the case that $X$ has smooth singular locus in \cite{tighe2022llv} when $X$ is symplectic}.  This implies the symplectic form induces isomorphisms \begin{equation}\label{equation intersection cohomology symplectic symmetry}
    L_\sigma^p: IH^{n-p,q}(X) \xrightarrow{\sim} IH^{n+p,q}(X),
\end{equation} for $0 \le p \le n$ \cite[Theorem 1.2]{tighe2022llv}

\subsection{Hodge theory of $\mathbb Q$-factorial terminal primitive symplectic 4-folds}

\begin{proposition} \label{proposition Q factorial H^2 IH^2}
Let $X$ be a $\mathbb Q$-factorial terminal primitive symplectic variety.  Then $$H^2(X, \mathbb Q) \xrightarrow{\sim} IH^2(X, \mathbb Q)$$ as (polarized) Hodge structures, and the inclusion $H^2(X, \mathbb Q) \hookrightarrow H^2(\widetilde X, \mathbb Q)$ has codimension $\dim H^0(E_{(1)})$.
\end{proposition}

\begin{proof}
This is \cite[Proposition 2.17]{tighe2022llv}, where the codimension of the inclusion $H^2(X, \mathbb Q) \hookrightarrow H^2(\widetilde X, \mathbb Q)$ is calculated in the proof.
\end{proof}

Next, we need to understand the weight filtration $W^\bullet$ on the cohomology groups $H^k(U, \mathbb Q)$.  Recall that a (rational) Hodge structure $V$ is \textit{Hodge-Tate} if $V^{p,q} = 0$ unless possibly $p = q$.

\begin{theorem} \label{theorem weight filtration coh. regular locus}
Let $X$ be a primitive symplectic variety with isolated singularities.  Then $$W_{k + 2}H^k(U, \mathbb Q) = 0$$ for each $k > 2n$, and $\mathrm{gr}_{k+1}^WH^k(U, \mathbb Q)$ is Hodge-Tate.
\end{theorem}

\begin{proof}
    For every $k$, we have a factorization \[ \begin{tikzcd}
    H^k(U, \mathbb C) \arrow{rr}{L_\sigma^p} \arrow{dr} && H^{k+2p}(U, \mathbb C) \\ 
    & H_c^{k+2p}(U, \mathbb C) \arrow{ur} &
    \end{tikzcd}\] of mixed Hodge structures of bidegree $(2p,0)$ (this is the content of \cite[\S 3]{tighe2022llv}).  On the Hodge filtration side, the top horizontal map induces a morphism $$L_\sigma^p: \mathrm{gr}_F^{n-p}H^k(U, \mathbb C) \to \mathrm{gr}_F^{n+p}H^{k+2p}(U, \mathbb C)$$ which by the degeneration of Hodge-to-de Rham on $U$ is an isomorphism if $k + 2p < 2n-1$ and is surjective if $k < 2n-1$.  On the weight filtration side, the morphism $$H_c^{k+2p}(U, \mathbb C) \to H^{k+2p}(U, \mathbb C)$$ lands in weight $k + 2p$ while the kernel of the morphism $H^k(U, \mathbb C) \to H_c^{k+2p}(U, \mathbb C)$ includes everything in weight $k+1$ since $H^k(U, \mathbb C)$ includes only higher weights and $H_c^{k+2p}(U, \mathbb C)$ is supported in lower weights.  We conclude that $\mathrm{gr}_F^{n-p}W_kH^k(U, \mathbb C) = 0$ if $k + 2p < 2n-1$ while $\mathrm{gr}_F^{2n+p}W_{k+2p + 1}H^{2k+p}(U, \mathbb C) = 0$ if $k < 2n-1$. This proves the vanishing of the weight $W_{k+2}$ piece of the mixed Hodge structure.  For the Hodge-Tate property, $W_k = 0$ for $k < 2n$, so assume that $k \ge 2n$.  Using the symplectic symmetry, for $2n \ge p > \frac{k + 1}{2}$ --- so that we are mapping across middle cohomology --- we get $\mathrm{gr}_F^pW_{k+1}H^k(U, \mathbb C) = 0$.
\end{proof}

\begin{corollary} \label{corollary weight filtration compactly supported cohomology}
If $X$ is a primitive symplectic variety with isolated singularities, then the weight filtration $W_l$ on $H^k_c(U, \mathbb Q)$ vanishes for $l \le k-1$ and $k < \dim X$, and $W_{k-1}H_c^k(U, \mathbb Q)$ is Hodge-Tate.
\end{corollary}

\begin{proof}
This follows from Poincar\'e duality and Theorem \ref{theorem weight filtration coh. regular locus}, noting that the weight filtration $W^lH^k = 0$ for $l < k$.   
\end{proof}

\begin{theorem} \label{theorem important vanishing Q fact term}
If $X$ is a $\mathbb Q$-factorial terminal primitive symplectic variety, then $$H^0(X, R^1\pi_*\Omega_{\widetilde X}^1(\log E)) = 0.$$
\end{theorem}

\begin{proof}
First, the Hodge filtration on the (pure) Hodge structure $H^2(U, \mathbb C)$ agrees with the filtration induced by the spectral sequence $$E_1^{p,q} = \mathbb H^q(X, \mathbf R\pi_*\Omega_{\widetilde X}^p(\log E)) \Rightarrow H^{p+q}(U, \mathbb C)$$ which degenerates at $E_1$.  In particular, $\mathbb H^1(X, \mathbf R\pi_*\Omega_{\widetilde X}^1(\log E)) = H^{1,1}(U)$ (this is just another way of stating (\ref{equation spectral sequence log regular locus})).  There is also a spectral sequence $$E_2^{p,q} = H^q(X, R^p\pi_*\Omega_{\widetilde X}^1(\log E)) \Rightarrow \mathbb H^{p+q} (X, \mathbf R\pi_*\Omega_{\widetilde X}^1(\log E)).$$ But $E_2^{p,q} = 0$ for $q > 0$ since $X$ has isolated singularities, and $E_2^{p,q} = 0$ for $p > 1$ by Theorem \ref{theorem local and steenbrink vanishing intro}.  Therefore, the spectral sequence degenerates at $E_2$ when $p + q \le 1$; thus $$H^1(\widetilde X, \Omega_{\widetilde X}^1(\log E)) \cong H^1(X, \pi_*\Omega_{\widetilde X}^1(\log E)) \oplus H^0(X, R^1\pi_*\Omega_{\widetilde X}^1(\log E)).$$ But $\pi_*\Omega_{\widetilde X}^1 \xrightarrow{\sim} \pi_*\Omega_{\widetilde X}^1(\log E)$ by (\ref{equation holomorphic extension}), and $H^1(X, \pi_*\Omega_{\widetilde X}^1) \cong H^1(\widetilde X, \Omega_{\widetilde X}^1(\log E))$ by Proposition \ref{proposition Q factorial H^2 IH^2}. 
\end{proof}

As a corollary of Theorem \ref{theorem main}, we also get the following:

\begin{corollary} \label{corollary other important vanishing Q fact term degree n-1}
If $X$ is a $\mathbb Q$-factorial terminal primitive symplectic variety, then $$H^0(X, R^1\pi_*\Omega_{\widetilde X}^{2n-1}(\log E)(-E)) = 0.$$
\end{corollary}

\begin{remark} \label{remark q factorial deformations}
The vanishing of Theorem \ref{theorem important vanishing Q fact term} has an interpretation to deformation theory.  Supposing the vanishing holds, the Leray spectral sequence implies $H^2(X, \mathbb Q) \xrightarrow{\sim} H^2(U, \mathbb Q)$.  The $(1,1)$-parts of these Hodge structures parametrize the locally-trivial and flat deformations of $X$, respectively, if $X$ is terminal, and these are known to agree in the $\mathbb Q$-factorial setting \cite[Main Theorem]{namikawa2005deformations}.  The vanishing of Corollary \ref{corollary other important vanishing Q fact term degree n-1} also has an interpretation in deformation theory and has been studied in work of Friedman-Laza for nice classes of singular Calabi-Yau varieties, see for example \cite{friedman2024kliminal}.  It would be interesting to know if there is a similar relationship between $\mathbb Q$-factorial terminal Calabi-Yau varieties and their deformations.
\end{remark}

\begin{corollary} \label{corollary not so important vanishing}
If $X$ is a $\mathbb Q$-factorial terminal primitive symplectic 4-fold, then $$H^0(R^1\pi_*\Omega_{\widetilde X}^2(\log E)(-E)) = 0.$$
\end{corollary}

\begin{proof}
We claim that the differentiation morphism $$d: R^1\pi_*\Omega_{\widetilde X}^2(\log E)(-E) \to R^1\pi_*\Omega_{\widetilde X}^3(\log E)(-E)$$ is an isomorphism for any terminal 4-fold, not necessarily $\mathbb Q$-factorial or primitive.  Indeed, this problem is local, and we may assume that $X$ is the germ of Stein open neighborhood.  In this case, we again consider the spectral sequence $$E_1^{p,q} = H^q(\widetilde X, \Omega_{\widetilde X}^p(\log E)(-E) \Rightarrow \mathbb H^q(\widetilde X, \Omega_{\widetilde X}^p(\log E)(-E))$$ of Corollary \ref{corollary extra vanishing}, which abuts to 0. But by Theorem \ref{theorem local and steenbrink vanishing intro}, the only terms that survive are $E_1^{q,0}$ for $0 \le q \le 4$, $E^{2,1}_1$ and $E_1^{3,1}$.  Thus, $$E_2^{3,1} \cong E_1^{3,1}/E_1^{2,1},$$ and $E_2^{p,q} = E_\infty^{p,q} = 0$ for any $p,q$.  Now the claim follows from Corollary \ref{corollary other important vanishing Q fact term degree n-1}. 
\end{proof}

\begin{remark}
Theorem \ref{theorem local and steenbrink vanishing intro}, Theorem \ref{theorem important vanishing Q fact term}, and Corollaries \ref{corollary other important vanishing Q fact term degree n-1}, \ref{corollary not so important vanishing} imply that most of the higher cohomology sheaves $R^q\pi_*\Omega_{\widetilde X}^p(\log E)(-E)$ vanish for $\mathbb Q$-factorial terminal 4-folds, and the ones that don't at least have no global sections.  We are not sure if these extra vanishing should descend to the level of sheaves, as we cannot rule out that a $\mathbb Q$-factorial terminal 4-fold does not admit a symplectic resolution upon passing to formal completion at a singular point, and Corollary \ref{corollary vanishing symplectic resolution} would imply the non-vanishing.  This happens, for example, for the singular symplectic moduli space of Kaledin-Lehn-Sorger \cite[Remark 6.3]{kaledin2005singular}, which are $\mathbb Q$-factorial terminal. The singular loci of these examples are not smooth, and it may be the case this does not occur for isolated singularities. 
\end{remark}

\begin{theorem} \label{theorem hodge structure h3}
Let $X$ be a $\mathbb Q$-factorial terminal primitive symplectic variety.  Then $H^3(X, \mathbb Q)$ carries a pure Hodge structure, and $$H^3(X, \mathbb Q) \xrightarrow{\sim} IH^3(X, \mathbb Q) \xrightarrow{\sim} H^3(\widetilde X, \mathbb Q),$$ where $\pi:\widetilde X \to X$ is any log-resolution of singularities.
\end{theorem}

\begin{proof}
Since $X$ has rational singularities, $R^1\pi_*\mathscr O_{\widetilde X} = 0$.  Therefore, $R^1\pi_*\mathbb Z = 0$ by the exponential sheaf sequence, and the Leray spectral sequence $$E_2^{p,q} = H^p(X, R^q\pi_*\mathbb Z) \Rightarrow H^{p+q}(\widetilde X, \mathbb Z)$$ induces a long-exact sequence \begin{equation}\label{equation long exact sequence H^3}
0 \to H^2(X, \mathbb Z) \to H^2(\widetilde X, \mathbb Z) \to H^0(X, R^2\pi_*\mathbb Z) \to H^3(X, \mathbb Z) \to H^3(\widetilde X, \mathbb Z) \to 0.
\end{equation} Now, the natural map $H^3(X, \mathbb Q) \to IH^3(X, \mathbb Q)$ has kernel equal to $W_2H^3(X, \mathbb Q)$.  In particular, (\ref{equation long exact sequence H^3}) implies at the very least $$\mathrm{gr}_3^WH^3(X, \mathbb Q) \xrightarrow{\sim} IH^3(X, \mathbb Q) \xrightarrow{\sim} H^3(\widetilde X, \mathbb Q).$$ However, we will show that $H^3(X, \mathbb C) \hookrightarrow IH^3(X, \mathbb C)$ using the Hodge filtration, which is enough.  Since $X$ has isolated singularities, we have $IH^3(X, \mathbb C) \xrightarrow{\sim} H^3(U, \mathbb C)$, and the Hodge filtration is the one induced by (\ref{equation spectral sequence log regular locus}).  By duality, the Hodge filtration on $H_c^3(U, \mathbb C) \cong H^3(X, \mathbb C)$ is given by the spectral sequence $$E_1^{p,q} = H^q(\widetilde X, \Omega_{\widetilde X}^p(\log E)(-E)) \Rightarrow H^{p+q}(X, \mathbb C).$$ 
 By Proposition \ref{proposition reflexivity and symmetry of h^0's} and (\ref{equation natural morphism from du bois to dual}), we have morphisms $$H^q(X, \pi_*\Omega_{\widetilde X}^p(\log E)) \to H^q(\widetilde X, \Omega_{\widetilde X}^p(\log E)(-E)) \to H^q(\widetilde X, \Omega_{\widetilde X}^p(\log E)).$$ 
We claim $H^q(X, \pi_*\Omega_{\widetilde X}^p(\log E)) \to H^q(X, \Omega_{\widetilde X}^p(\log E))$ is injective for $p + q =3$.  It is obvious if $q = 0$, $q = 3$ holds since $X$ has rational singularities, and $q = 1$ holds by Leray.  To get $q = 2$, we use the fact that $H^0(X, R^1\pi_*\Omega_{\widetilde X}^1(\log E)) = 0$ by Theorem \ref{theorem important vanishing Q fact term}.  This proves that $H^3(X, \mathbb C) \hookrightarrow H^3(U, \mathbb C)$, and $H^3(X, \mathbb Q)$ carries a pure Hodge structure.
\end{proof}

Next, we need to understand the Hodge theory of log-resolutions of singularities.  The following is a slight generalization of a result of Kaledin, which states that the fibers of \textit{symplectic} resolutions are Hodge-Tate:

\begin{theorem}
Let $X$ be a $\mathbb Q$-factorial terminal primitive symplectic 4-fold, and let $E$ be the exceptional divisor of some log-resolution $\pi:\widetilde X \to X$.  Then $H^k(E, \mathbb Q)$ is pure of Hodge-Tate type for each $k$.  Moreover, $H^4(X, \mathbb Q)$ carries a pure Hodge structure.
\end{theorem}

\begin{proof}
For $k \ge 4$, $H^k(E, \mathbb Q)$ carries a pure Hodge structure.  In fact, we have by the decomposition theorem $$H^k(\widetilde X, \mathbb Q) = IH^k(X, \mathbb Q)\oplus V,$$ where $V = H^k(E, \mathbb Q)$ for $k \ge 4$ and $H_k(E, \mathbb Q)$ for $k < 4$.  By duality and Theorem \ref{theorem hodge structure h3}, $IH^5(X, \mathbb Q) \xrightarrow{\sim} H^5(\widetilde X, \mathbb Q)$, whence $H^5(E, \mathbb Q) = 0$.  It is obvious that $H^6(E, \mathbb Q)$ and $H^0(E, \mathbb Q)$ are Hodge-Tate.  We continue on with the rest of the computation.

For a primitive symplectic variety, there is an exact sequence $$0 \to H^1(E, \mathbb Q) \to H^2_c(U, \mathbb Q) \to H^2(\widetilde X, \mathbb Q) \to H^2(E, \mathbb Q) \to H^3_c(U, \mathbb Q) \to H^3(\widetilde X, \mathbb Q),$$  as $H^1(\widetilde X, \mathbb Q) = 0$ by (\ref{equation hodge to de rham X}).  Moreover, $H^k_c(U, \mathbb Q) \xrightarrow{\sim} H^k(X, \mathbb Q)$ for $k \ge 2$ since $X$ has isolated singularities.  This implies we get a short exact sequence $$0 \to H^2_c(U, \mathbb Q) \to H^2(\widetilde X, \mathbb Q) \to H^2(E, \mathbb Q) \to 0,$$ and $H^1(E, \mathbb Q) = 0$.  We also have $H^2(E, \mathbb Q)$ is pure of Hodge-Tate type, since $H_c^2(U, \mathbb C) \to H^2(\widetilde X, \mathbb C)$ is an isomorphism outside the $(1,1)$-part.

Next, we actually show $H^3(E_{(1)}, \mathbb Q) = 0$.  It is classically known $H^{3,0}(E_{(1)}) = 0$, but this easily follows from the fact that $R^{3}\pi_*\mathscr O_{\widetilde X} = 0$.  It is enough to prove that $H^{2,1}(E_{(1)}) = 0$.  If we consider the logarithmic sheaf sequence, we only need to show that
            \begin{enumerate}
                \item $H^1(\widetilde X, \Omega_{\widetilde X}^3) \to H^1(\widetilde X, \Omega_{\widetilde X}^3(\log E))$ is surjective, and
                \item $H^2(\widetilde X, \Omega_{\widetilde X}^3) \to H^2(\widetilde X, \Omega_{\widetilde X}^3(\log E))$ is injective.
            \end{enumerate}
            The second part follows from the fact $H^3(X, \mathbb Q) \to H^3(\widetilde X, \mathbb Q)$ is surjective by (\ref{equation long exact sequence H^3}), which by Poincar\'e duality implies $H^5(\widetilde X, \mathbb Q) \hookrightarrow H^5(U, \mathbb Q)$ since $H_c^3(U, \mathbb Q) \xrightarrow{\sim} H^3(X, \mathbb Q)$.  Therefore, $H^3(U, \mathbb Q)$ carries a pure Hodge structure, and the isomorphism follows from (\ref{equation spectral sequence log regular locus}).  Now we need to prove the first part. Consider the long exact sequence \begin{equation}
            \begin{split}
                 H^1_E(\widetilde X, \Omega_{\widetilde X}^3(\log E)) \to & H^1(\widetilde X, \Omega_{\widetilde X}^3(\log E)) \to H^1(U, \Omega_U^3) \\&\to H^{2}_E(\widetilde X, \Omega_{\widetilde X}^3(\log E)).   
            \end{split}
        \end{equation} We claim that the first and last terms vanish.  Indeed, the groups $H^q_E(\widetilde X, \Omega_{\widetilde X}^p(\log E))$ are dual to $$H^{4-q}(\widetilde X_E, \Omega_{\widetilde X}^{4-p}(\log E)(-E)),$$ where $\widetilde X_E$ is the completion of $X$ along $E$.  Therefore, they can be computed by a spectral sequence $$E_2^{r,s} = H^r(X_\Sigma, R^s\pi_*\Omega_{\widetilde X}^{4-p}(\log E)(-E)) \Rightarrow H^{r+s}(\widetilde X_E, \Omega_{\widetilde X}^{4-p}(\log E)(-E),$$ where $X_\Sigma$ is the completion of $X$ along $\Sigma$.  Since $\dim \Sigma \le 0$, it follows that \begin{equation}\label{equation isomorphism for local cohomology middle}
        H^0(X_\Sigma, R^{4-q}\pi_*\Omega_{\widetilde X}^{4-p}(\log E)(-E))^* \cong H^q_E(\widetilde X, \Omega_{\widetilde X}^p(\log E).\end{equation} But $R^3\pi_*\Omega_{\widetilde X}^1(\log E)(-E) = 0$ by (\ref{equation steenbrink vanishing extended}) and $$R^2\pi_*\Omega_{\widetilde X}^3(\log E)(-E) = 0$$ by Steenbrink vanishing (\ref{equation steenbrink vanishing general}).  Therefore, $H^1(\widetilde X, \Omega_{\widetilde X}^3(\log E)) \xrightarrow{\sim} H^1(U, \Omega_U^3)$. Since $U$ is symplectic, there is an isomorphism $H^1(U, \Omega_U^1) \xrightarrow{\sim} H^1(U, \Omega_U^3)$, and the former group is isomorphic to $IH^{1,1}(X)$.  Therefore, $H^1(U, \Omega_U^3) \cong IH^{3,1}(X)$ by \cite[Theorem 1.2]{tighe2022llv}. But by the decomposition theorem, we also have that $H^{3,1}(\widetilde X) \to IH^{3,1}(X)$ is surjective.  This finishes the proof that $H^3(E_{(1)}, \mathbb Q) = 0$.

Now we prove $H^3(E, \mathbb Q) = 0$.  There is similarly an exact sequence \begin{equation} \label{equation exact sequence H^3 and H^4}
      0 \to H^3(E, \mathbb Q) \to H_c^4(U, \mathbb Q) \to H^4(\widetilde X, \mathbb Q).  
    \end{equation} We do not yet know that $H_c^4(U, \mathbb Q)$ is a pure Hodge structure, but we do understand \textit{part} of the Hodge filtration on its dual $H^4(U, \mathbb Q)$.  Following the identification (\ref{equation isomorphism for local cohomology middle}) a bit further, we see that $$H^{4-p}(\widetilde X, \Omega_{\widetilde X}^p(\log E)) \to H^{4-p}(U, \Omega_U^p)$$ is at least surjective for $p \ge 1$.  Recall that we have seen $H^1(\widetilde X, \Omega_{\widetilde X}^3) \to H^1(U, \Omega_U^1)$ is surjective, and $H^2(\widetilde X, \Omega_{\widetilde X}^2) \to H^2(\widetilde X, \Omega_{\widetilde X}^2(\log E))$ is surjective. Using that $U$ is symplectic again, we also have $H^3(U, \Omega_U^1) \xrightarrow{\sim} H^3(U, \Omega_U^3)$.  By duality, $H^6(U, \mathbb Q)$ carries a pure Hodge structure.  Moreover, $$H^3(\widetilde X, \Omega_{\widetilde X}^3(\log E)) \to H^3(U, \Omega_U^3)$$ is an isomorphism by (\ref{equation isomorphism for local cohomology middle}), using the fact that $R^1\pi_*\Omega_{\widetilde X}^1(\log E)(-E) = 0$ by Theorem \ref{theorem local and steenbrink vanishing intro}. It follows immediately that $H^3(\widetilde X, \Omega_{\widetilde X}^1) \hookrightarrow H^3(U, \Omega_U^1)$, and in fact they are isomorphic since they have the same dimension. Therefore, the morphism $$\mathrm{gr}_p^FH^4(\widetilde X, \mathbb C) \to \mathrm{gr}_p^F H^4(U, \mathbb Q)$$ is surjective for $p \ge 1$, and dually $$\mathrm{gr}_p^F H^4_c(U, \mathbb C) \to \mathrm{gr}_p^F H^4(\widetilde X, \mathbb C)$$ is injective for $p \ge 1$.  It follows that $\mathrm{gr}_p^FH^3(E, \mathbb C) = 0$ for $p \ge 1$.  But $\mathrm{gr}_F^0H^3(E, \mathbb Q) = H^3(E, \mathscr O_E) = 0$, either by our discussion above or classically since $X$ is Cohen-Macaulay, and so $$H^3(E, \mathbb C) = \bigoplus_{p \ge 0} \mathrm{gr}_p^F H^3(E, \mathbb C) = 0.$$  This also proves that $H_c^4(U, \mathbb Q) = H^4(X, \mathbb C)$ carries a pure Hodge structure. 
\end{proof}

\begin{corollary}
If $X$ is a $\mathbb Q$-factorial terminal primitive symplectic 4-fold, then $$H^k(E_{(m)}, \mathbb Q) = 0$$ for odd $k$.
\end{corollary}

\begin{proof}
This follows from the Hodge theory of $H^k(E, \mathbb Q)$.  Since $E$ has simple normal crossings, the cohomology of $E$ can be computed by a complex $$0 \to H^k(E_{(1)}, \mathbb Q) \to H^k(E_{(2)}, \mathbb Q) \to ... \to H^k(E_{(m)}, \mathbb Q) \to ...$$ In general, the cohomology of this complex computes the weight filtration on $H^k(E, \mathbb Q)$, but in our case these are all pure; in particular, $$b_k(E) = \sum_m(-1)^{m}b_k(E_{(m)}).$$ One quickly checks that $b_5(E_{(1)}) = b_1(E_{(1)}) = 0$, $b_3(E_{(2)}) = b_1(E_{(2)}) = 0$ since we have already shown that $b_3(E) = b_3(E_{(1)}) = 0$, and $b_1(E_{(3)}) = 0$.  
\end{proof}

\begin{remark}
We do not claim that the cohomology groups of the $E_{(m)}$ are Hodge-Tate for $m > 1$.  In fact, it seems necessary that $H^{2,0}(E_{(1)}) \ne 0$ if $E$ has normal crossing singularities.
\end{remark}

\section{Some Results on the Intersection Hodge Module} \label{section intersection module}

\subsection{Symmetries for the intersection Hodge module} Recall that the Du Bois complex is a formal consequence of Deligne's construction of a mixed Hodge structure on the cohomology of $X$.  Saito constructs this data from the perspective of mixed Hodge modules in \cite{saito1990mixed}.  Specifically, there is an object $\mathbb Q_X^H \in D^b\mathrm{MHM}(X)$ in the bounded derived category of mixed Hodge modules on $X$, whose underlying $\mathbb Q$-complex is $\mathbb Q_X$.  It is defined as the pullback $$\mathbb Q_X^H : = r_X^*\mathbb Q_{\mathrm{pt}}$$ where $r_X: X \to \mathrm{pt}$ is the structure morphism. 

Let $\mathcal M_X^\bullet$ be the underlying complex of (left) $\mathscr D_Y$-modules. If $\mathrm{DR}_X(-)$ is the de Rham functor, it follows from \cite[Theorem 4.2]{saito2000mixed} that there is a canonical isomorphism \begin{equation}\label{equation hodge module du bois}
    \underline \Omega_X^p[-p] \xrightarrow{\sim} \mathrm{gr}_F^p\mathrm{DR}(\mathcal M_X^\bullet),
\end{equation} where $F$ is the (Hodge) filtration associated to $\mathbb Q_X^H$.

Being a complex of mixed Hodge modules, the cohomology objects $H^k\mathbb Q_X^H$ come with a weight filtration $W$ such that $\mathrm{gr}_i^WH^k\mathbb Q_X^H$ is a pure Hodge module.  By \cite[(4.5.2)]{saito1990mixed}, the complex $\mathbb Q_X^H$ is of weight $\le 0$, which means $\mathrm{gr}_i^WH^k\mathbb Q_X^H= 0$ for $i > k$.  This allows us to connect $\mathbb Q_X^H$ to the intersection pure Hodge module, denoted by $\mathrm{IC}_X$, which is the unique extension of the variation $\mathbb Q_U[2n]$ of the regular locus $U$.  It is a pure Hodge module of weight $2n$, and we write $\mathcal{IC}_X$ for its underlying $\mathscr D_Y$-module structure with Hodge filtration $F$.  By definition, the cohomology groups $\mathbb H^{k-2n}(X, \mathrm{IC}_X)$ are the intersection cohomology groups $IH^k(X, \mathbb Q)$.  By \cite[(4.5.9)]{saito1990mixed}, we have an identification \begin{equation}
    \mathrm{gr}_{2n}^WH^{2n}\mathbb Q_X^H = \mathrm{IC}_X.
\end{equation} On the holomorphic side, this reads as $\mathrm{DR}(\mathrm{gr}_{2n}^WH^{2n}M_X^\bullet) \xrightarrow{\sim} \mathrm{DR}(\mathcal{IC}_X)$. Since $\mathbb Q_X^H$ is of weight $\le 0$, this provides a natural quotient morphism $\mathbb Q_X^H \to \mathrm{IC}_X[-2n]$ in $D^b\mathrm{MHM}(X)$.  Since $\mathrm{IC}_X$ is a pure Hodge module, it is self dual up to a shift; in particular, this gives the natural factorization \begin{equation}
    \mathbb Q_X^H \to \mathrm{IC}_X[-2n] \to \mathbb D_X(\mathbb Q_X^H)(-2n)[-4n].
\end{equation} 

We conclude with the following theorem:

\begin{theorem} \label{theorem symmetry intersection cohomology}
If $X$ is a symplectic variety of dimension $2n$, then the symplectic form induces quasi-isomorphisms $$L_\sigma^{n-1}: \mathrm{gr}_{-1}^F\mathrm{DR}(\mathcal{IC}_X)[2n-1] \xrightarrow{\sim} \mathrm{gr}_{-(2n-1)}^F\mathrm{DR}(\mathcal{IC}_X)[1].$$
\end{theorem}

\begin{proof}
The existence of such a morphism follows from the fact that the symplectic form defines a map $\mathbb C_U[2n] \to \mathbb C_U[2n + 2]$ on the regular locus $U$, since $\mathrm{IC}_X$ is the unique extension of $\mathbb C_U[2n]$ to $X$.

The proof is nearly identical to Theorem \ref{lemma local vanishing isolated singularities}.  For brevity, let $K_q = \mathrm{gr}^F_{-q}\mathrm{DR}(\mathcal{IC}_X)[q-2n]$.  We get the following analogue of (\ref{equation commutative diagram}), using the fact that the complexes $K_{n-p}$ and $K_{n+p}$ are related by duality:

\begin{equation}
\begin{tikzcd}
    \mathbf R\mathscr Hom_{\mathscr O_X}(\tau_{\ge 1}K_{n+p}, \mathscr O_X) \arrow{d} \arrow{r}{(\tau_{\ge 1}L_\sigma^p)^*}~~~~~ & ~~~~~ \mathbf R\mathscr Hom_{\mathscr O_X}(\tau_{\ge 1} K_{n-p}, \mathscr O_X) \arrow{d} \\
    K_{n-p} \arrow{r}{(L_\sigma^p)^*} \arrow{d} & K_{n+p} \arrow{d}\\
    \mathbf R\mathscr Hom_{\mathscr O_X}(\mathscr H^0K_{n+p}, \mathscr O_X) \arrow{r}{(\mathscr H^0L_\sigma^p)^*} & \mathbf R\mathscr Hom_{\mathscr O_X}(\mathscr H^0K_{n-p}, \mathscr O_X),
    \end{tikzcd}
    \end{equation} By \cite[Proposition 8.1]{kebekus2021extending}, $\mathscr H^0K_{q} \cong \pi_*\Omega_{\widetilde X}^{q}$ for each $q$; in particular, $(\mathscr H^0L_\sigma^p)^*$ is a quasi-isomorphism.  By the (semi)perverse data of the Hodge module, also have $\mathscr H^jK_{n+p} = 0$ for $j > n-p$, and a similar argument to the proof of Theorem \ref{theorem main} shows $\mathscr H^jK_{n-p} = 0$ for $n-p < j < n+p$, and $L_\sigma^p$ is a quasi-isomorphism if and only if the vanishing holds for $j \ge n+p$.  For $p = n-1$, this extra vanishing holds by \cite[Proposition 8.4]{kebekus2021extending}.
\end{proof}

By taking hypercohomology, we get the following corollary:

\begin{corollary}
If $X$ is a primitive symplectic variety of dimension $2n$, the symplectic form induces isomorphisms $$IH^{1,q}(X) \xrightarrow{\sim} IH^{2n-1,q}(X)$$ for each $q$.
\end{corollary}

\begin{remark}
One of the key features of a compact hyperk\"ahler manifold $X$ is the structure the cohomology ring $H^*(X, \mathbb Q)$ inherits as a $\mathfrak g$-representation, where $\mathfrak g$ is the \textit{total Lie algebra}, or \textit{Looijenga-Lunts-Verbitsky} (LLV) algebra.  By definition, the algebra is generated by operators $$\mathfrak g = \langle L_\omega, \Lambda_\omega, H\rangle$$ where $L_\omega$ is the cupping morphism of a class $\omega \in H^2(X, \mathbb Q)$ which satisfies the Hard Lefschetz theorem $$L_\omega^k: H^{2n-k}(X, \mathbb Q) \xrightarrow{\sim} H^{2n+k}(X, \mathbb Q), \quad 0 \le k \le 2n,$$ and $H = (k-2n)\mathrm{id}$ is the unique weight operator on $H^k(X, \mathbb Q)$ which satisfies $[L_\omega, \Lambda_\omega] = H$.  By \cite{looijenga1997lie} and \cite{verbitsky1996cohomology}, the total Lie algebra takes on the form $$\mathfrak g \cong \mathfrak{so}((H^2(X, \mathbb Q), q_X)\oplus \mathfrak h),$$ where $q_X$ is the BBF-form and $\mathfrak h$ is a copy of the hyperbolic plane.  Both proofs rely on the underlying hyperk\"ahler metric.

The author gave an algebraic proof of this in \cite{tighe2022llv} by showing that the necessary commutator relations can be obtained from the symplectic isomorphisms $L_\sigma^p: H^{n-p,q}(X) \to H^{n+p,q}(X)$.  This gives a blueprint for extending this theorem to the singular setting using \textit{intersection} cohomology.  In particular, the author proves in \cite{tighe2022llv} that the natural extension of the LLV algebra to the intersection cohomology module $IH^*(X, \mathbb Q)$ of a primitive symplectic variety with \textit{isolated} singularities admits the structure $$\mathfrak{g} \cong \mathfrak{so}((IH^2(X, \mathbb Q), Q_X)\oplus \mathfrak h),$$ where $Q_X$ is a natural extension of the BBF form under the inclusion $H^2(X, \mathbb Q) \hookrightarrow IH^2(X, \mathbb Q)$.  The proof uses the fact that the maps $L_\sigma^p: IH^{n-p,q}(X) \to IH^{n+p,q}(X)$ are isomorphisms for each $p$ \cite[Theorem 3.5]{tighe2022llv}, using some global results from the Hodge theory of isolated singularities.  Therefore, Theorem \ref{theorem symmetry intersection cohomology} can be seen as a local version of this proof.
\end{remark}

\subsection{Deformation theory and intersection cohomology} As we pointed out in Remark \ref{remark q factorial deformations}, there is a connection between the Du Bois complex and deformations of symplectic varieties.  The deformation theory of primitive symplectic varieties has been extensively studied by work of Namikawa \cite{namikawa2000deformation}, \cite{namikawa2005deformations} and Bakker-Lehn \cite{bakker2016global}, \cite{bakker2018global}.  To conclude this paper, we give a slightly new perspective involving intersection cohomology, which we believe could be useful in more general settings.

Let $X$ be a primitive symplectic variety with universal deformation space $\mathrm{Def}(X)$.  There is a subspace $\mathrm{Def}^{\mathrm{lt}}(X) \subset \mathrm{Def}(X)$ parametrizing locally trivial deformations of $X$.  The tangent spaces at the reference point are $\mathrm{Ext}^1(\Omega_X^1, \mathscr O_X)$ and $H^1(X, T_X)$, respectively, where $T_X$ is the tangent sheaf of $X$, and the deformation spaces are compatible with the inclusion $H^1(X, T_X) \hookrightarrow \mathrm{Ext}^1(\Omega_X^1, \mathscr O_X)$ obtained by the local-to-global spectral sequence for Ext. The space of locally trivial deformations is unobstructed \cite[Theorem 4.7]{bakker2018global} while the full deformation space is unobstructed if $X$ is terminal \cite[Theorem (2.5)]{namikawa2000deformation}. 

We have seen there is a natural, injective morphism $H^2(X, \mathbb Q) \to IH^2(X, \mathbb Q)$.  Since $T_X \cong \Omega_X^{[1]}$, this gives an injection $$H^1(X, T_X) \hookrightarrow IH^{1,1}(X),$$ using the fact that $T_X \cong \Omega_X^{[1]}$.  We also have the following:

\begin{proposition} \label{proposition deformation intersection cohomology}
Let $X$ be a projective primitive symplectic variety.  There is an injective map $IH^{1,1}(X) \to \mathrm{Ext}^1(\Omega_X^1, \mathcal O_X)$ which depends only on $X$.
\end{proposition}

\begin{proof}
First, recall that $\mathscr H^{-1}\mathrm{gr}_{-(2n-1)}^F\mathrm{DR}(\mathcal{IC}_X) \cong \Omega_X^{[2n-1]} \cong \Omega_X^{[1]}$. By duality, this gives a morphism $$\mathrm{gr}_{-1}^F\mathrm{DR}(\mathcal{IC}_X)[2n-1] \to \mathbf R\mathscr Hom_{\mathscr O_X}(\Omega_X^{[1]}, \mathscr O_X[2n]).$$ Taking hypercohomology, we obtain a sequence of morphisms \begin{equation} \label{equation composition ext and intersection cohomology}
 IH^{1,1}(X) \to \mathrm{Ext}^1(\Omega_X^{[1]}, \mathcal O_X) \to \mathrm{Ext}^1(\Omega_X^1, \mathcal O_X)   
\end{equation} which depend only on $X$. To see it is injective, we pass to a $\mathbb Q$-factorial terminalization $\phi: Z \to X$.  By \cite[Theorem 1]{namikawa2005deformations}, there is a finite morphism $\mathrm{Def}(Z) \to \mathrm{Def}(X)$.  Since $Z$ is $\mathbb Q$-factorial terminal, this implies an isomorphism $$H^{1,1}(Z) \xrightarrow{\sim} \mathrm{Ext}^1(\Omega_X^1, \mathscr O_X).$$  On the other hand, the decomposition theorem implies an injection $$IH^{1,1}(X) \hookrightarrow IH^{1,1}(Z) \cong H^{1,1}(Z)$$ by Proposition \ref{proposition Q factorial H^2 IH^2}.  It follows that (\ref{equation composition ext and intersection cohomology}) is injective.
\end{proof}

In particular, we have a factorization \[\begin{tikzcd}
H^1(T_X) \arrow{dr} \arrow[swap]{rr} & & \mathrm{Ext}^1(\Omega_X^1, \mathcal O_X) \\
 & IH^{1,1}(X) \arrow{ur}& 
\end{tikzcd}.
\] There is actually a clear way of interpreting $IH^{1,1}(X)$ in terms of deformations.  Let $\Sigma$ be the singular locus.  There is a locus $\Sigma_0 \subset \Sigma$ such that $\mathrm{codim}_X(\Sigma_0) \ge 4$, and every point of $X\setminus \Sigma_0$ is locally a transversal slice of a rational double point.  By \cite[Lemma 1]{durfee1995intersection}, $IH^2(X, \mathbb Q) \cong IH^2(X\setminus \Sigma_0, \mathbb Q)$.  By definition, $X\setminus \Sigma_0$ is a $\mathbb Q$-homology manifold, and $IH^2(X\setminus \Sigma_0, \mathbb Q) \cong H^2(X\setminus \Sigma_0, \mathbb Q)$.  In particular, $H^2(X\setminus \Sigma_0, \mathbb Q)$ carries a pure Hodge structure, and in fact $$H^{1,1}(X\setminus \Sigma_0, \mathbb C) = H^1(X\setminus \Sigma_0, T_{X\setminus \Sigma_0}).$$  In particular, $IH^{1,1}(X)$ parametrizes the deformations of $X$ which are locally trivial in codimension 2.  Therefore:

\begin{enumerate}
    \item $IH^{1,1}(X) = \mathrm{Ext}^1(\Omega_X^1, \mathscr O_X)$ if $X$ is terminal.

    \item $H^{1,1}(X) \cong IH^{1,1}(X)$ if $X$ is $\mathbb Q$-factorial terminal (this is a different way of looking at Proposition \ref{proposition Q factorial H^2 IH^2}).  
\end{enumerate} As a consequence of Theorem \ref{theorem symmetry intersection cohomology} and the above perspective, we get the following curious condition related to Remark \ref{remark q factorial deformations}:

\begin{corollary}
Let $X$ be a primitive symplectic variety of dimension $2n$.  If $$\mathscr H^{-(2n-2)}\mathrm{gr}_{-1}^F\mathrm{DR}(\mathcal{IC}_X) = 0,$$ then the deformations of $X$ are locally trivial. 
\end{corollary}

\begin{proof}
By Theorem \ref{theorem symmetry intersection cohomology}, the only non-trivial cohomology sheaves $\mathscr H^j\mathrm{gr}_{-1}^F\mathrm{DR}(\mathcal{IC}_X)$ are for $j = -(2n-1)$ and $-(2n-2)$.  If the latter vanishes, then $$\mathrm{gr}_{-1}^F\mathrm{DR}(\mathcal{IC}_X) \cong_{\mathrm{qis}} \mathscr H^{-(2n-1)}\mathrm{gr}_{-1}^F\mathrm{DR}(\mathcal{IC}_X) \cong \pi_*\Omega_{\widetilde X}^1,$$ the last isomorphism being \cite[Proposition 8.1]{kebekus2021extending}, and $IH^{1,1}(X) \cong H^{1,1}(X)$.  This proves the claim.
\end{proof}

It would be interesting to know if there is a similar condition involving intersection cohomology and the deformation theory of Calabi-Yau varieties.

\bibliography{bib.bib}
\bibliographystyle{alpha}

\end{document}